\documentclass[10pt]{article}
\usepackage{color}
\usepackage{amssymb}
\usepackage{amsthm,array,amssymb,amscd,amsfonts,latexsym, url}
\usepackage{amsmath}
\usepackage[all]{xy}

\newtheorem{theo}{Theorem}[section]

\newtheorem{prop}[theo]{Proposition}
\newtheorem{lemm}[theo]{Lemma}
\newtheorem{coro}[theo]{Corollary}
\newtheorem{rema}[theo]{Remark}
\newtheorem{Defi}[theo]{Definition}

\newtheorem{conj}[theo]{Conjecture}

\title{Remarks and questions on coisotropic subvarieties and $0$-cycles of hyper-K\"ahler varieties}
\author{Claire Voisin
\\CNRS, Institut de Math\'ematiques de Jussieu and IAS \footnote{This research has been   supported by  The Charles Simonyi Fund and The Fernholz Foundation}}

\date{}

\newfont{\gothic}{eufb10}
\begin{document}
\maketitle
\setcounter{section}{-1}

\section{Introduction}

This paper proposes a conjectural picture for the structure of the Chow ring  ${\rm CH}^*(X)$
 of a (projective) hyper-K\"ahler variety $X$, that seems to emerge
 from the recent papers \cite{chapa}, \cite{shenvial}, \cite{vial}, \cite{voisinlaz}, with emphasis on the Chow group ${\rm CH}_0(X)$
of $0$-cycles (in this paper,  Chow groups will be taken with
$\mathbb{Q}$-coefficients). Our motivation is Beauville's conjecture
(see \cite{beausplit}) that for such an $X$, the Bloch-Beilinson
filtration has a natural, multiplicative, splitting. This statement
is hard to make precise since the Bloch-Beilinson filtration is not
known to exist, but for $0$-cycles, this means that $${\rm
CH}_0(X)=\oplus {\rm CH}_0(X)_i,$$ where the decomposition is given
by the action of self-correspondences $\Gamma_i$ of $X$, and where
the group ${\rm CH}_0(X)_i$ depends only on $(i,0)$-forms on $X$
(the correspondence $\Gamma_i$ should act as $0$ on $H^{j,0}$ for
$j\not=i$, and $Id$ for $i=j$). We refer to the paragraph \ref{secBB}
at the end of this introduction
for the axioms of the Bloch-Beilinson filtration
and  we will refer to it  in Section \ref{sec3}
when providing some evidence for our conjectures.
 Note that a hyper-K\"ahler variety
$X$ has $H^{i,0}(X)=0$ for odd $i$, so
 the Bloch-Beilinson filtration  $F_{BB}$ has to satisfy
$F_{BB}^{i}{\rm CH}_0(X)=F^{i+1}_{BB}{\rm CH}_0(X)$ when $i$ is odd.
Hence we are only interested in the $F^{2i}_{BB}$-levels, which we
denote by ${F'}^i_{BB}$.
Note also that there are concrete consequences of the Beauville conjecture that can be attacked
directly, namely, the $0$-th piece ${\rm CH}(X)_0$ should map isomorphically via the cycle class map
to its image in $H^*(X,\mathbb{Q})$ which should be
the subalgebra $H^*(X,\mathbb{Q})_{alg}$ of algebraic cycle classes, since the Bloch-Beilinson
filtration is conjectured to have $F^1_{BB}{\rm CH}^*(X)={\rm CH}^*(X)_{hom}$.
Hence there should be a subalgebra of ${\rm CH}^*(X,\mathbb{Q})$ which
is isomorphic to the subalgebra $H^{2*}(X,\mathbb{Q})_{alg}\subset
H^{2*}(X,\mathbb{Q}) $ of algebraic classes. Furthermore,
this subalgebra has to contain
${\rm NS}(X)={\rm Pic}(X)$. Thus a concrete subconjecture is the following
prediction (cf. \cite{beausplit}):
\begin{conj} \label{conjBV} (Beauville) Let $X$ be a projective hyper-K\"ahler manifold. Then
the cycle class map is injective on the subalgebra of ${\rm CH}^*(X)$
 generated by divisors.
\end{conj}
 This conjecture has been enlarged in \cite{voisinPAMQ} to include
 the Chow-theoretic Chern classes of $X$, $c_i(X):=c_i(T_X)$ which should thus be thought
 as being  contained in the $0$-th piece of the conjectural
 Beauville decomposition. Our purpose in this paper is to introduce a new set
 of classes which should also be put in this $0$-th piece, for example, the constant cycles
 subvarieties of maximal dimension (namely $n$, with ${\rm dim}\,X=2n$, because they
 have to be isotropic, see Section \ref{sec2}) and their
 higher dimensional generalization, which are algebraically coisotropic.
 Let us explain the motivation for this, starting from
 the study of $0$-cycles.

Based on the case of $S^{[n]}$ where we have the results of \cite{beauvoisin},
\cite{ogrady}, \cite{voisinlaz}
that concern the ${\rm CH}_0$ group of a $K3$ surface
but will be reinterpreted in a slightly different form in Section
\ref{sec1},
we introduce the following decreasing filtration $S_\cdot$ on
${\rm CH}_0(X)$ for any hyper-K\"ahler manifold (the definition can in fact be introduced
for any algebraic variety but we think it is interesting only in the hyper-K\"ahler case).

First of all, recall that the orbit $O_x$ of a point $x\in X$ under rational equivalence, defined
as
$$O_x=\{y\in X,\,y\equiv_{rat} x\,\,{\rm in}\,\,X\},$$
is a countable union of closed algebraic subsets of $X$. It thus has a dimension, which is the
supremum of the dimensions of the closed algebraic subsets of $X$ appearing
in this union. If ${\rm dim}\,O_x=i$, there exists a subvariety $Z\subset
X$ of dimension $i$, such that all points of $Z$ are rationally equivalent
to $x$ in $X$. The variety $Z$ is by definition   a ``constant cycle subvariety'' in the sense of Huybrechts
\cite{huy}.

\begin{Defi} \label{defifilt} We define
  $S_iX\subset X$ to be the set of points in
$X$ whose orbit under rational equivalence has dimension $\geq i$.

The filtration $S_\cdot$ is then defined  by letting $S_i{\rm CH}_0(X)$ be the
subgroup of ${\rm CH}_0(X)$ generated by classes of points $x\in S_iX$.
\end{Defi}
 The set $S_iX\subset X$ is  a
countable union
of closed algebraic subsets of $X$ and by definition
$S_i{\rm CH}_0(X)$ is the $\mathbb{Q}$-vector subspace
of ${\rm CH}_0(X)$  generated by the points in $S_iX$.

Let us explain the relationship between this definition and the one introduced  by O'Grady
in \cite{ogrady}. O'Grady introduces a decreasing filtration
$S_{OG}$ on the ${\rm CH}_0$ group of a $K3$ surface $S$, defined (up to a shift
of indices) by
\begin{eqnarray}
\label{eqfiltogk31}
S_{OG,i}{\rm CH}_0(S)_d=\{z\in {\rm CH}_0(S)_d,\,z\equiv_S z'+i\,o_S,\,z'\in S^{(d-i)}\}.
\end{eqnarray}
Here, ${\rm CH}_0(S)_d$ is the set of $0$-cycles of degree $d$
modulo rational equivalence on $S$, and $o_S\in {\rm CH}_0(S)_1$ is
the ``Beauville-Voisin'' canonical $0$-cycle of $S$ introduced in
\cite{beauvoisin}. The symbol $\equiv_S$ means ``rationally
equivalent in $S$''. A variant of this definition where we replace
rational equivalence
 in $S$ by rational equivalence in $S^{[n]}$ (or equivalently $S^{(n)}$) provides a filtration
$N_\cdot$ on ${\rm CH}_0(S^{[n]})$, namely
\begin{eqnarray}\label{eqfiltog}
N_i{\rm CH}_0(S^{[n]})={\rm Im}\,({(io_S)}_*:{\rm CH}_0(S^{[n-i]})\rightarrow
{\rm CH}_0(S^{[n]})).
\end{eqnarray}
This filtration exists for any surface $S$ equipped with a
base-point. It depends however on the choice of the  point, or at
least of the rational equivalence class of the point. Here,
the only specificity of $K3$ surfaces is thus the fact that there is a
canonically defined rational equivalence class of a point, namely
the  canonical zero-cycle $o_S$. It is proved in \cite[Theorem
1.4]{voisinlaz} that an equivalent definition of O'Grady's
filtration (\ref{eqfiltogk31}) on ${\rm CH}_0(S)$ can be given as
follows (here we are assuming $i\geq 0$ and the assumption ``${\rm
dim}\geq 0$'' means in particular ``non-empty''):
\begin{eqnarray} \label{eqfiltogK3} S_{OG,i}{\rm CH}_0(S)_d=\{z\in {\rm CH}_0(S)_d,\,{\rm dim}\,O_z^S\geq i\}.
\end{eqnarray}
Here $O_z^S\subset S^{(d)}$ is the orbit of $z$ for  rational equivalence in $S$, that is
$$O_z^S=\{z'\in S^{(d)},\,z'\equiv_S z\}.$$
Note that this is different from the orbit $O_z$ of
$z$ as a point of $S^{(n)}$ or $S^{[n]}$ for {\it rational equivalence in} $S^{[n]}$.
One has however the obvious inclusion $O_z\subset O_z^S$ which will be exploited in this paper.
As we will prove in Section
\ref{sec1},  the  main result in \cite{voisinlaz}
also implies
\begin{theo} \label{theointrofiltK3}(Cf. Theorem \ref{theofiltcomp}) The filtration $S_\cdot$ introduced in Definition
\ref{defifilt} coincides with
the filtration $N_\cdot$ of (\ref{eqfiltog}) on $0$-cycles on $S^{[n]}$, when
$S$ is a $K3$ surface and $o_S$ is a point representing the canonical $0$-cycle of $S$.
\end{theo}
On the other hand, for any surface $S$ and choice of point $o_S$, the  filtration
$N_\cdot$
 provides a splitting of the Bloch-Beilinson filtration
on ${\rm CH}_0(S^{[n]})$ (see Proposition \ref{prodecompchow} in Section \ref{sec1}).
Hence, when $S$ is a $K3$ surface and $o_S$ is the canonical $0$-cycle,
the filtration $S_\cdot$  provides a splitting of the Bloch-Beilinson filtration
on ${\rm CH}_0(S^{[n]})$.
Our  hope and guiding idea in this paper is that this is the general situation
for hyper-K\"ahler manifolds.
A first concrete  conjecture in this direction is the following:
\begin{conj}\label{conjmain}  Let $X$ be projective hyper-K\"ahler manifold of dimension
$2n$. Then for any nonnegative integer $i\leq n$,
the set
$$S_iX:=\{x\in X,\,{\rm dim}\,O_x\geq i\}$$
has dimension $2n-i$.
\end{conj}
The case $i=n$, that is Lagrangian constant cycles subvarieties, was first asked by Pacienza
(oral communication).
This conjecture and the axioms of Bloch-Beilinson filtration   would imply in particular that
the natural map
$$S_i{\rm CH}_0(X)\rightarrow {\rm CH}_0(X)/{F'}_{BB}^{n-i+1}{\rm CH}_0(X)$$
is surjective (see Lemma \ref{lesurjtoopp}). We conjecture in fact that this map is
an isomorphism (cf. Conjecture \ref{conjfiltopp}).

A good evidence for Conjecture \ref{conjmain}
is provided by the results of Charles and Pacienza
\cite{chapa}, which deal with the deformations of $S^{[n]}$ (case $i=1$), and
the deformations of $S^{[2]}$, (case $i=2$), and Lin \cite{lin} who constructs
constant cycles Lagrangian subvarieties in hyper-K\"ahler manifolds admitting
a Lagrangian fibration. Another evidence is given by the
complete family
of hyper-K\"ahler $8$-folds constructed by Lehn-Lehn-Sorger-van Straten in \cite{LLS}
 that we will study in Section \ref{sec4} (see Corollary
\ref{corollss}). We will prove there that they satisfy Conjecture \ref{conjmain}. In fact, we
describe a parametrization of them which should make accessible for them a number
of conjectures made in this paper, \cite{beausplit}, or \cite{voisinPAMQ}, by reduction
to the case of the variety of lines of a cubic fourfold.

We will explain in Section \ref{sec2} that Conjecture \ref{conjmain}
 contains
as a by-product
an existence conjecture for algebraically coisotropic
(possibly singular) subvarieties of $X$ of codimension $i$. By this we mean the following:
\begin{Defi}\label{deficoiso} A subvariety $Z\subset X$ is coisotropic if
for any $z\in Z_{reg}$, $T_{Z,z}^{\perp \sigma}\subset T_{Z,z}$.
\end{Defi}
Here $\sigma$ is the $(2,0)$-form on $X$.
Given a coisotropic subvariety $Z\subset X$, the open set
$Z_{reg} $ has an  integrable distribution (a foliation) given by the
vector subbundle $T_Z^{\perp \sigma}$, with  fiber
$T_{Z,z}^{\perp \sigma}\subset T_{Z,z}$, or equivalently,  the kernel
of the restricted form $\sigma_{\mid Z}$ which has by assumption the constant minimal rank.
\begin{Defi}\label{deficoisoalg} A subvariety $Z\subset X$ of codimension $i$ is algebraically coisotropic if
the distribution above is algebraically integrable, by which we mean
that there exists
a rational map $\phi:Z\dashrightarrow B$ onto a variety $B$ of dimension $2n-2i$
such that $\sigma_{\mid Z}$ is the pull-back to $Z$ of a
$(2,0)$-form on $B$, $\sigma_{\mid Z}=\phi^*\sigma_B$.
\end{Defi}
Any divisor in a hyper-K\"ahler variety is coisotropic. However, only few of them are algebraically isotropic: In fact,
Amerik and Campana prove in \cite{amca} that if $n\geq2$, a smooth divisor is algebraically isotropic if and only if it is uniruled. The regularity assumption here is of course crucial.

The link between Conjecture \ref{conjmain} and the existence of algebraically
coisotropic subvarieties is provided by Mumford's theorem \cite{mumford} on pull-backs of
holomorphic forms and rational equivalence. The following result
 will be proved in Section \ref{sec2}, where
we will also describe the restrictions satisfied by the cohomology classes of coisotropic subvarieties.
\begin{theo} (Cf. Theorem \ref{propcoiso}) Let $Z$ be a codimension $i$ subvariety of a
hyper-K\"ahler manifold $X$. Assume that any point of $Z$ has an
orbit of dimension $\geq i$ under rational equivalence in $X$ (that
is $Z\subset S_iX$). Then $Z$ is algebraically coisotropic and the
fibers of the isotropic fibration are  $i$-dimensional orbits of $X$
for rational equivalence.
\end{theo}
 In Section \ref{sec3}, starting from the case of $S^{[n]}$, where things work very well thanks to Theorem \ref{theointrofiltK3}, we will then  discuss
the following next ``conjecture'':
\begin{conj} \label{conjfiltopp} Let $X$ be a projective hyper-K\"ahler manifold of dimension
$2n$. Then  the filtration $S_\cdot$ is opposite to the filtration $F'_{BB}$ and thus provides
a splitting of it.
\end{conj}
Concretely, this means that
\begin{eqnarray}\label{formulatrans}
S_i{\rm CH}_0(X)\cong {\rm CH}_0(X)/{F'}_{BB}^{n-i+1}{\rm CH}_0(X)
\end{eqnarray}
for any $i\geq 0$.
Assuming (\ref{formulatrans}) holds, we have a natural decomposition
of ${\rm CH}_0$ into a  direct sum
$${\rm CH}_0(X)=\oplus_j{\rm CH}_0(X)_{2j},$$
where ${\rm CH}_0(X)_{2j}:=S_{n-j}{\rm CH}_0(X)\cap F^{2j}_{BB} {\rm
CH}_0(X)$, and this gives a splitting of the Bloch-Beilinson
filtration.

We will explain in Section \ref{sec3} how this conjecture would fit with the expected multiplicativity
property of the Beauville filtration, and in particular with
the following expectation:

\begin{conj} The classes of codimension $i$ subvarieties of $X$ contained in $S_iX$
belong to the $0$-th piece of the Beauville decomposition and their cohomology classes
generate the space of coisotropic classes.
\end{conj}

Again, this has concrete consequences that can be investigated for themselves and independently
of the existence of a Bloch-Beilinson filtration, namely the fact
that the cycle class map is injective on the subring of ${\rm CH}(X)$ generated
by these classes and divisor classes.

The paper is organized as follows: In section \ref{sec2}, we will describe
the link between families of constant cycles subvarieties and algebraically coisotropic
subvarieties.
In section \ref{sec1}, we will compare the filtrations $N_\cdot$ and  $S_{\cdot}$
for $X={\rm Hilb}(K3)$. Section \ref{sec3} is devoted to stating conjectures
needed to construct a Beauville decomposition starting from  the filtration
$S_\cdot$. Finally, Section \ref{sec4} will provide a number of geometric constructions
and various evidences for these conjectures, in three cases: Hilbert schemes of $K3$ surfaces,
Fano varieties of lines of cubic fourfolds, and finally the more recent
$8$-folds constructed in \cite{LLS} starting from the Hilbert scheme of cubic rational
curves on cubic fourfolds.

\vspace{0.5cm}

{\bf Thanks}. I thank F. Charles,  K. O'Grady, G. Pacienza and G. Sacc\`a for interesting discussions
related to this paper.
\subsection{Bloch-Beilinson filtration \label{secBB}}
The Bloch-Beilinson filtration $F^\cdot_{BB}$ on the Chow groups with
$\mathbb{Q}$-coefficients of  smooth projective varieties
 is a decreasing filtration which is subject to the following axioms:
\begin{enumerate}
\item \label{itemcorr} It is preserved by correspondences: If $\Gamma\in {\rm CH}(X\times Y)$, then
for all $i$'s, $\Gamma_*(F^i{\rm CH}(X))\subset F^i{\rm CH}(Y)$.
\item \label{itemF1} $F^0{\rm CH}(X)={\rm CH}(X)$,  $F^1{\rm CH}(X)={\rm CH}(X)_{hom}$.
\item \label{itemmult} It is multiplicative: $F^i{\rm CH}(X)\cdot F^j{\rm CH}(X)\subset F^{i+j}{\rm CH}(X)$,
where $\cdot$ is the intersection product.
\item $F^{i+1}{\rm CH}^i(X)=0$ for all $i$ and $X$.
\end{enumerate}
Note that items \ref{itemF1} and \ref{itemmult} imply that
a correspondence $\Gamma\in {\rm CH}(X\times Y)$ which is cohomologous to $0$ shifts the
Bloch-Beilinson filtration:
\begin{eqnarray}\label{eqGammaaction}\Gamma_*(F^i{\rm CH}(X))\subset F^{i+1}{\rm CH}(Y).
\end{eqnarray}
One can also be more precise at this point, namely asking that
$Gr_F^i{\rm CH}^l(X)$ is governed only by $H^{2l-i}(X)$ and its Hodge structure.
Then (\ref{eqGammaaction}) becomes:
\begin{eqnarray}\label{eqGammaaction2}\Gamma_*(F^i{\rm CH}^l(X))\subset F^{i+1}{\rm CH}^{l+r}(Y),
\end{eqnarray}
if $[\Gamma]_*:H^{2l-i}(X,\mathbb{Q})\rightarrow H^{2l+2r-i}(Y,\mathbb{Q})$ vanishes (so we are considering only one K\"unneth component of $[\Gamma]\in H^{2r+2n}(X\times Y,\mathbb{Q})$, where
$n={\rm dim}\, X,\,\Gamma\in {\rm CH}^{r+n}(X\times Y)$).
The reason why, assuming a Bloch-Beilinson filtration exists, the graded pieces of it for $0$-cycles
depend only on holomorphic forms is the fact that according
to (\ref{eqGammaaction}), $Gr_F^i{\rm CH}_0(X)$ should be governed by
the cohomology group $H^{2n-i}(X,\mathbb{Q})$ or dually $H^i(X,\mathbb{Q})$.
On the other hand, ${\rm CH}_0(X)$ is not sensitive to the cohomology of
$X$ which is of geometric coniveau $\geq 1$, that is, supported on a divisor,
hence assuming the generalized Hodge conjecture holds,
it should be sensitive only to the group $H^i(X,\mathbb{Q})/C^1H^i(X,\mathbb{Q})$, where
$N^1H^i(X,\mathbb{Q})$ is the maximal Hodge substructure of $H^i(X,\mathbb{Q})$ which is
of Hodge coniveau $\geq1$, that is, contained in $F^1H^i(X,\mathbb{C})$. But
clearly $H^i(X,\mathbb{Q})/N^1H^i(X,\mathbb{Q})=0$ if and only if $H^{i,0}(X)=0$.

We will refer to the set of axioms  above in the form ``if a Bloch-Beilinson filtration exists'', with
the meaning that it exists for all $X$ (this is necessary as the axiom
\ref{itemcorr} is essential).

Note that there exist many  varieties for which there are
natural candidates for  the Bloch-Beilinson filtration (for example
surfaces) but apart from curves and more generally varieties with representable Chow groups, none are known to satisfy the full set of axioms
above.
\section{Constant cycles subvarieties and coisotropic classes \label{sec2}}

Let $X$ be a smooth projective variety over $\mathbb{C}$ and
let $f:Z\rightarrow X$   be a morphism from a smooth
projective variety admitting a surjective morphism
$p:Z\rightarrow B$, where $B$ is smooth, with the following property:

(*) {\it The fibers of $p$ map via $f$ to
``constant cycle'' subvarieties},

that is, all points in a given fiber of
$p$ map via $f$ to rationally equivalent points  in $X$.
Mumford's theorem \cite{mumford} then implies:
\begin{lemm} \label{propmum} Under assumption (*), for any holomorphic
form $\alpha$ on $X$, there is a holomorphic form
$\alpha_B$ on $B$ such that $f^*\alpha=p^*\alpha_B$.
\end{lemm}
\begin{proof} Let $B'\subset Z$ be a closed subvariety such that
$p_{\mid B'}=:p'$ is generically finite and let $N:={\rm deg}\,B'/B$, $f':=f_{\mid B'}$.
We have two correspondences between $Z$ and $X$, namely $\Gamma_f$,
which is  given by the graph
of $f$, and $\Gamma'$, which is defined as the
composition with $p$ of the correspondence
$$\Gamma'':=\frac{1}{N}(p',Id)_*(\Gamma_{f'})$$ between $B$ and $X$.
Assumption (*) says that
$$\Gamma_*=\Gamma'_*:{\rm CH}_0(Z)\rightarrow {\rm CH}_0(X).$$

Mumford's theorem then says that for any
holomorphic form $\alpha$ on $X$, one has
$$f^*\alpha=\Gamma_f^*\alpha={\Gamma'}^*\alpha=({\Gamma''\circ p})^*\alpha=p^*( {\Gamma''}^*\alpha),$$
which proves the result with $\alpha_B={\Gamma''}^*\alpha$.
\end{proof}

We now consider the case where $X$ is a projective hyper-K\"ahler manifold of dimension
$2n$ with holomorphic
$2$-form $\sigma$.
A particular case of Lemma \ref{propmum} is the case where $B$ is a point,
which gives the following statement:
\begin{coro}\label{coroiso} Let $\Gamma\subset X$ be a constant cycle subvariety.
Then $\Gamma$ is an isotropic subvariety, that is
$$\sigma_{\mid \Gamma_{reg}}=0.$$
In particular, ${\rm dim}\,\Gamma\leq n$, and in the case of equality, $\Gamma$ is a Lagrangian
(possibly singular) subvariety.
\end{coro}
We are interested in this paper to coisotropic subvarieties whose study started
only recently (see \cite{amca}, \cite{manetti}). Such subvarieties can be constructed
applying the following result:

\begin{theo}\label{propcoiso} Assume $S_iX$ (see Definition
\ref{defifilt}) contains a  closed algebraic subvariety
$Z$ of codimension
$\leq i$. Then:

(i)  The codimension of $Z$ is equal to $i$
and (the smooth locus of) $Z$ is algebraically coisotropic (see Definition \ref{deficoisoalg}).

(ii) Furthermore,  the general fibers of the coisotropic fibration $Z\dashrightarrow B$ are constant cycles
subvarieties of $X$ of dimension $i$.
\end{theo}
\begin{proof} By assumption, for any $z\in Z$, there is a subvariety
$K_z\subset X$ which is contained in $O_z$ and has dimension $i$.
Using the countability of Hilbert schemes, there exists
a generically finite cover
$\alpha: Z'\rightarrow Z\subset X$, and a family $p: K\rightarrow Z'$ of varieties of dimension $i$ over $Z'$,
together with a morphism
 $f:K\rightarrow X$ satisfying the following properties:
\begin{enumerate}
 \item\label{itema} For any $k\in K$, $f(k) $ is rationally equivalent to $\alpha\circ p(k)$ in $X$.

 \item \label{itemb}  The morphism $f$ restricted to the general fiber
 of $p$ is generically finite over (or even birational to) its image in $X$.
\end{enumerate}
 Property \ref{itema} and Lemma \ref{propmum} imply that
 \begin{eqnarray}\label{equaform}
 f^*\sigma= p^*(\alpha^*\sigma) \,\,{\rm in}\,\,H^0(K,\Omega_K^2).
 \end{eqnarray}

 Formula (\ref{equaform}) tells us that the $2$-form $f^*\sigma$
 has the property that for the general point $k\in K$,
 the tangent space $T_{K_k}$ to the fiber $K_k$ of $p$ passing through
 $k$
is  contained in the  kernel of    $f^*\sigma_{\mid k}\in \Omega_{K,k}^2$.
Equivalently, the vector space $f_*(T_{K_k})\subset T_{X,f(k)}$ is contained in the kernel of
the form $\sigma_{\mid f(K)}$ at the point $f(k)$. From now on, let us denote
$Z'':=f(K)\subset X$.
By assumption \ref{itemb}, $f_*(T_{K_k})$ has dimension $i$ for general $k$, and because
$\sigma$ is nondegenerate, this implies that the rank of the map
$f$ at $k$ is at most $2n-i$. By Sard's theorem, it follows
that $Z''$ has dimension $\leq 2n-i$ and that the generic rank of the $2$-form
$\sigma_{\mid Z''}$ is at most $2n-2i$.
Hence the rank of the $2$-form $f^*\sigma=p^*(\alpha^*\sigma)$ on
$K$ is at most $2n-2i$, and as $p$ is dominating and $\alpha$ is generically finite,
this implies that the rank of $\sigma_{\mid Z_{reg}}$ is
$\leq 2n-2i$. Hence $Z$ satisfies ${\rm codim}\,Z\geq i$,
hence in fact ${\rm codim}\,Z=i$, and furthermore the rank of
$\sigma_{\mid Z_{reg}}$ is equal to $2n-2i$, so
that $Z$ is a coisotropic subvariety of $X$.

Let us now prove that $Z$ is algebraically coisotropic. We proved above that the
varieties $f(K_k)\subset Z''$ have their tangent space
contained in the kernel of the form $\sigma_{\mid Z''}$.
Let now
$\Gamma\subset K$ be a subvariety which is generically finite over $Z'$ and
dominates $Z''$.
Such a $\Gamma$ exists since we proved that ${\rm dim}\,Z''\leq 2n-i={\rm dim}\,Z'$.
Denote by
$$f_\Gamma:\Gamma\rightarrow Z'',\,
q_\Gamma:\Gamma\rightarrow Z$$
the restrictions to
$\Gamma$ of $f$ and $\alpha\circ p$ respectively.
Restricting
(\ref{equaform}) to $\Gamma$ gives
\begin{eqnarray}\label{equaformGamma}
 f_\Gamma^*(\sigma_{\mid Z''})= q_\Gamma^*(\sigma_{\mid Z}) \,\,{\rm in}\,\,H^0(\Gamma,\Omega_\Gamma^2).
 \end{eqnarray}
 The varieties $f_\Gamma^{-1}(f(K_k))$ are  tangent to
 the kernel of the form $f_\Gamma^*(\sigma_{\mid Z''})$, hence of the form
 $ q_\Gamma^*(\sigma_{\mid Z})$, which means that their images
 $$q_\Gamma(f_\Gamma^{-1}(f(K_k)))\subset Z$$
  are tangent
 to the kernel of the form $\sigma_{\mid Z}$. As they are (for general $k$) of dimension $\geq i$ because $q_\Gamma$ is generically finite,
 and as ${\rm dim}\,Z=2n-i$,
 one concludes that they are in fact of dimension $i$, and are thus
 algebraic  integral leaves
 of the distribution on $Z_{reg}$ given by
 ${\rm Ker}\,\sigma_{\mid Z}$.
One still needs to explain why this is enough to imply
 that $Z$ is algebraically coisotropic. We already proved that
 $Z$ is swept-out by algebraic varieties $Z_t$ which are $i$-dimensional
 and tangent to the distribution on $Z_{reg}$ given by
 ${\rm Ker}\,\sigma_{\mid Z}$. We just have to construct
 a dominant rational map $Z\dashrightarrow B$ which admits the $Z_t$'s as fibers. However,
 we observe that if $B$ is the Hilbert scheme parameterizing $i$-dimensional subvarieties
 of $Z$ tangent to this distribution along $Z_{reg}$, and $M\rightarrow B$ is the universal family of such subvarieties, the morphism $M\rightarrow X$ is birational since there is  a unique
 leaf of the distribution at any point of $Z_{reg}$. This provides us with the desired
 fibration. This proves (i).

(ii)  We have to prove that the fibers of the coisotropic fibration of $Z$ are
constant cycle subvarieties of $X$ of dimension $i$.
By construction, they are the varieties $q_\Gamma(f_\Gamma^{-1}(f(K_k)))$, $k\in K$.
But $f(K_k)$ is by definition a constant cycle subvariety of dimension $i$ of $X$,
all of whose  points are rationally equivalent  to $\alpha\circ p(k)$,
by  condition  \ref{itema}. It follows from condition \ref{itema} again
that all points in $q_\Gamma(f_\Gamma^{-1}(f(K_k)))$ are rationally equivalent
in $X$ to $f(k)$.

\end{proof}
\subsection{Classes of coisotropic subvarieties \label{sec11}}
This subsection is devoted to the description of
the restrictions on the cohomology classes of coisotropic subvarieties of
a  hyper-K\"ahler manifold $X$ of dimension $2n\geq 4$.
More precisely, we will only study those classes which can be written
as a polynomial in divisor classes and the class $c\in S^2H^2(X,\mathbb{Q})\subset H^4(X,\mathbb{Q})$ defined as follows:
The Beauville-Bogomolov form $q$ on $H^2(X,\mathbb{R})$, which is characterized up to a coefficient
 by the condition
that for any $\lambda\in H^2(X,\mathbb{Q})$,
\begin{eqnarray}\label{eqq} \int_X\lambda^{2n}=\mu_Xq(\lambda)^n
\end{eqnarray}
for some nonzero constant $\mu_X$,
is nondegenerate.
The form $q$  provides an invertible symmetric map
$$ H^2(X,\mathbb{Q})\rightarrow H^2(X,\mathbb{Q})^*$$
with inverse
$$c\in {\rm Hom}_{sym}(H^2(X,\mathbb{Q})^*, H^2(X,\mathbb{Q}))=S^2H^2(X,\mathbb{Q})
\subset H^4(X,\mathbb{Q}).$$
It is also easy to check that
for every nonnegative integer $i\leq n$, there exists a nonzero constant $\mu_{i,X}$ such that
for any $\lambda\in H^2(X,\mathbb{Q})$
\begin{eqnarray}\label{eqqc} \int_Xc^i\lambda^{2n-2i}=\mu_{i,X}q(\lambda)^{n-i}.
\end{eqnarray}
Our goal in this section is to compute the ``coisotropic classes''
which can be written as polynomials $P(c,l_j)$, where $l_j\in {\rm NS}^2(X)$.
The following lemma justifies Definition \ref{defcoisoclass} of a coisotropic class:

\begin{lemm}\label{leisoclass} Let $Z\subset X$ be a codimension $i$ subvariety
and let $[Z]\in H^{2i}(X,\mathbb{Q})$ its cohomology class.
Then $Z$ is coisotropic if and only
\begin{eqnarray}\label{eqcoisoclass}
[\sigma]^{n-i+1}\cup [Z]=0\,\,{\rm in}\,\,H^{2n+2}(X,\mathbb{C}),
\end{eqnarray}
\end{lemm}
where $[\sigma]\in H^2(X,\mathbb{C})$ is the cohomology class of $\sigma$.
\begin{proof} Indeed, $Z$ is coisotropic if and only if the restricted form
$\sigma_{\mid Z}$ has rank $2n-2i$ on $Z_{reg}$,
which is equivalent to the vanishing
\begin{eqnarray}
\label{eqvan1}
\sigma^{n-i+1}_{\mid Z_{reg}}=0\,\,{\rm in}\,\,H^0(Z_{reg},\Omega^2_{ Z_{reg}}).
\end{eqnarray}
 We claim that
this vanishing in turn is equivalent to
the condition
\begin{eqnarray}
\label{eqvan2}[\sigma]^{n-i+1}\cup [Z]=0\,\,{\rm in}\,\,H^{2n+2}(X,\mathbb{C}).
\end{eqnarray}
 Indeed, (\ref{eqvan1}) clearly implies (\ref{eqvan2}). In the other direction, we can assume
 $i\geq2$ since the for $i=1$ there is nothing to prove (all divisor classes are
 coisotropic).
 The vanishing
of $[\sigma]^{n-i+1}\cup [Z]$  in $H^{2n+2}(X,\mathbb{C})$
then implies that  the
cup-product
$l^{i-2}\cup [\sigma]^{n-i+1}\cup[\overline{\sigma}]^{n-i+1}\cup [Z]$
vanishes in  $H^{4n}(X,\mathbb{C})$, where $l$ is the first Chern class
of a very  ample line bundle $L$ on $X$,  so that (\ref{eqvan1}) implies the
vanishing of the integral
$\int_{Z'}  \sigma^{n-i+1}\wedge \overline{\sigma}^{n-i+1}$, where
$Z'\subset Z$ is the complete intersection of $i-2$ general members in $|L|$ and
the integral has to be understood as an integral on a desingularization of $Z'$.
But the form  $\sigma^{n-i+1}_{\mid Z'}\wedge \overline{\sigma}^{n-i+1}_{\mid Z'}$ can be written
(along $Z'_{reg}$) as
$\pm f\,\nu$ where $\nu $ is a volume form on $Z'_{reg}$, and
the continuous function $f$ is real nonnegative.

It follows that  the vanishing of the integral implies that $\sigma^{n-i+1}_{\mid Z'}=0$
and as the tangent space of $Z'$ at a given point
 can be chosen arbitrarily (assuming $L$ ample enough), this implies
that $\sigma^{n-i+1}=0$. (Of course, if $Z$ is smooth, we can directly apply
the second Hodge-Riemann relations to $Z$.)

\end{proof}
We thus make the following definition
\begin{Defi}\label{defcoisoclass} A coisotropic class on $X$ of degree $2i$ is a  Hodge class
$z$
of degree $2i$ which satisfies the condition
\begin{eqnarray}
\label{eqvan3}[\sigma]^{n-i+1}\cup z=0\,\,{\rm in}\,\,H^{2n+2}(X,\mathbb{C}).
\end{eqnarray}
\end{Defi}
The contents of Lemma \ref{leisoclass} is thus that  an {\it effective} class is the class of a coisotropic subvariety if and only if it is coisotropic.
\begin{rema}\label{rema28dec}{\rm It is not known if the class $c$ is algebraic. However it is known to be algebraic
for those $X$ which are deformations of ${\rm Hilb}^n(K3)$ (see \cite{mark}). In general however, the class $c_2(X)$ is of course always algebraic and its projection to $S^2H^2(X,\mathbb{Q})$ is
a nonzero multiple of $c$. (This projection is well-defined, using the canonical
decomposition $H^4(X,\mathbb{Q})\cong S^2H^2(X,\mathbb{Q})\oplus S^{2n-2}H^2(X,\mathbb{Q})^{\perp}$.)
}
\end{rema}
Formula (\ref{eqvan3}) in Definition \ref{defcoisoclass} may give the feeling that the space of coisotropic
classes depends on the period point $[\sigma]\in H^2(X,\mathbb{C})$.
This is not true, as we are going to show, at least if we restrict to classes which can be written as
polynomials $P(c,l_j)$ involving only divisor classes and the class $c$, and $X$ is very general in moduli.

To state the next theorem, we need to introduce some notation.
 Let $H^2(X,\mathbb{Q})_{tr}\subset H^2(X,\mathbb{Q})$
be the orthogonal complement of ${\rm NS}(X)_\mathbb{Q}={\rm Hdg}^2(X,\mathbb{Q})$
with respect to the Beauville-Bogomolov form $q$, and let
$QH^2(X,\mathbb{C})_{tr}\subset H^2(X,\mathbb{Q})_{tr}$ be the quadric defined by $q$.
The $\mathbb{Q}$-vector space $H^2(X,\mathbb{Q})$ splits as the orthogonal
direct sum $$H^2(X,\mathbb{Q})_{tr}\oplus^{\perp}{\rm NS}(X)_\mathbb{Q},$$ and
it is immediate to check that $c\in S^2H^2(X,\mathbb{Q})$ decomposes
as
$$c=c_{tr}+c_{alg},$$
where $c_{tr}\in S^2H^2(X,\mathbb{Q})_{tr}$ and $c_{alg}\in S^2{\rm NS}(X)_\mathbb{Q}$.

If $X$ is a projective hyper-K\"ahler manifold, let $\rho:=\rho(X)$.
 We will say that $X$ is very general if  $X$ corresponds to a very general
point in the family of deformations of $X$ preserving ${\rm NS}(X)$, that is, with
Picard number at least $\rho$.
\begin{theo} \label{theoisoclass} Assume  $X$ is very general; then the following hold.

(i) A Hodge class $z\in {\rm Hdg}^{2i}(X,\mathbb{Q})$ is coisotropic if and only if
\begin{eqnarray}
\label{eqcoisoavecalpha}\alpha^{n-i+1}\cup z=0\,\,{\rm in}\,\,H^{2n+2}(X,\mathbb{C})
\end{eqnarray}
 for any
$\alpha\in QH^2(X,\mathbb{C})_{tr}$.

(ii) If $z=P(c,l_j)$ is a polynomial as above,
$z$ is coisotropic if and only if, for any
$\alpha\in QH^2(X,\mathbb{C})_{tr}$, for any $\beta\in H^2(X,\mathbb{Q})$,
\begin{eqnarray}
\label{eqcoisoavecalphabeta} \alpha^{n-i+1}\cup \beta^{n-1}\cup  z=0\,\,{\rm in}\,\,H^{4n}(X,\mathbb{C})\stackrel{\int_X}{\cong}\mathbb{C}.
\end{eqnarray}

(iii)  If $\rho(X)=1$, the space of coisotropic classes which are polynomials
$P(c,l),\,l\in {\rm NS}(X)$, is of dimension $\geq 1$, for any $0\leq i\leq n$.
If $\rho(X)=2$, it has dimension $\geq i+1$.

(iv)  If $l$ is an ample class on $X$, a nonzero coisotropic class
$$z=\lambda_0l^i+\lambda_1 l^{i-2}c_{tr}+\ldots+\lambda_jl^{i-2j} c_{tr}^j,\,\,j:=\lfloor i/2\rfloor$$
has $\lambda_0\not=0$. In particular, the space of such classes has dimension $1$.
\end{theo}
\begin{proof} (i) As $X$ is very general, its period point
$[\sigma]$ is a very general point of $QH^2(X,\mathbb{C})_{tr}$ and
it follows that the Mumford-Tate group $MT(X)$ of the Hodge
structure on $H^2(X,\mathbb{Q})_{tr}$ is equal to the orthogonal
group $SO(H^2(X,\mathbb{Q})_{tr},q)$. Denote by
$<QH^2(X,\mathbb{C})_{tr}^{n-i+1}>$ the complex vector subspace of
$S^{n-i+1}H^2(X,\mathbb{C})_{tr}$ generated by $x^{n-i+1}$ for all
$x$ satisfying $q(x)=0$. This vector space is defined over
$\mathbb{Q}$, that is, it is the complexification
 of a $\mathbb{Q}$-vector subspace
 $$<QH^2(X,\mathbb{Q})_{tr}^{n-i+1}>\subset S^{n-i+1}H^2(X,\mathbb{Q})_{tr},$$ and it is in fact a sub-Hodge structure
of $S^{n-i+1}H^2(X,\mathbb{Q})_{tr}$. This Hodge structure is simple under our assumption. Indeed,
the Mumford-Tate group $MT(X)$ is the full orthogonal
group $SO(H^2(X,\mathbb{Q})_{tr},q)$
and $<QH^2(X,\mathbb{C})_{tr}^{n-i+1}>$ is an irreducible representation of $SO(H^2(X,\mathbb{C})_{tr},q)$. This implies a fortiori that $<QH^2(X,\mathbb{Q})_{tr}^{n-i+1}>$ is an irreducible representation of $SO(H^2(X,\mathbb{Q})_{tr},q)$, and the simplicity of the Hodge structure
follows since by definition of the Mumford-Tate group $MT(X)$, sub-Hodge structures correspond to sub-representations of $MT(X)$. The proof of (i) is now immediate. First of all, by (\ref{eqvan3}),
a Hodge class $z$ of degree $2i$ on $X$
 is a coisotropic class if and only if
\begin{eqnarray}
\label{eqencorecoisosigma}
[\sigma]^{n-i+1}\cup z=0\,\,{\rm in}\,\,H^{2n+2}(X,\mathbb{C}).
\end{eqnarray}
An equivalent way of stating this property is to say that the class
$$[\sigma]^{n-i+1}\in <QH^2(X,\mathbb{C})_{tr}^{n-i+1}>$$ is annihilated
by the morphism of Hodge structures

\begin{eqnarray}
\label{eqecupz}\cup z: <QH^2(X,\mathbb{Q})_{tr}^{n-i+1}>\rightarrow H^{2n+2}(X,\mathbb{Q}).
\end{eqnarray}
As the Hodge structure on the left is simple, this morphism is either injective or identically
$0$, so the vanishing (\ref{eqencorecoisosigma}) is equivalent to the vanishing
of $\cup z$ in (\ref{eqecupz}).

\vspace{0.5cm}

(ii) We know by (i) that
$z$ is coisotropic if and only if $\alpha^{n-i+1}\cup z=0\,\,{\rm in}\,\,H^{2n+2}(X,\mathbb{Q})$
for any $\alpha\in QH^2(X,\mathbb{C})_{tr}$.
On the other hand, as the class $c$ belongs to $S^2H^2(X,\mathbb{Q})$, the class
$z=P(c,l_j),\,l_j\in\,\,{\rm NS}(X)$, belongs to the image of $S^iH^2(X,\mathbb{Q})$ in $H^{2i}(X,\mathbb{Q})$.
We now use the results of \cite{verb} which say that the subalgebra of $H^*(X,\mathbb{Q})$
 generated by $H^2(X,\mathbb{Q})$ is Gorenstein, that is self-dual with respect to Poincar\'e duality.
Hence a class $\alpha^{n-i+1}\cup z$, with $z=P(c,j_j)$, vanishes if and only if its cup-product with any
class in $S^{n-1}H^2(X,\mathbb{Q})\subset H^{2n-2}(X,\mathbb{Q})$ vanishes, which
is exactly statement (ii).

\vspace{0.5cm}

(iii)  As shown by Fujiki \cite{fujiki}, it follows  from formula (\ref{eqqc}) that
for any $l,\,e\in H^2(X,\mathbb{Q})$, and any
polynomial $P(c,l,e)$ of weighted degree $i$, there exists
a polynomial
$R$ depending only on $P$, in the variables
$q(\alpha,l),\,q(\alpha,e),\,q(\alpha),\,q(\beta,l),\,q(\beta,e),\,q(\beta),q(\alpha,\beta)$,
such that
for any $\alpha,\,\beta\in H^2(X,\mathbb{Q})$,
\begin{eqnarray}\label{eqfujita} \int_X\alpha^{n-i+1}\cup \beta^{n-1}\cup P(c,l,e)=R(q(\alpha,l),q(\alpha,e),q(\alpha),q(\beta,l),q(\beta,e),q(\beta),q(\alpha,\beta)).
\end{eqnarray}
Assume now that $\alpha\in QH^2(X,\mathbb{Q})_{tr}$ and $l,\,e$ form
a basis of ${\rm NS}(X)$ (so $\rho(X)=2$). Then
$$q(\alpha)=0,\,\,q(\alpha,e)=0,\,\,
q(\alpha,l)=0$$
 and thus (\ref{eqfujita}) becomes
 \begin{eqnarray}\label{eqfujitaalpha} \int_X\alpha^{n-i+1}\cup \beta^{n-1}\cup P(c,l,e)=R_0(q(\beta,l),q(\beta,e),q(\beta),q(\alpha,\beta)),
\end{eqnarray}
where $R_0$ is the restriction of $R$ to the subspace where the first three coordinates
vanish.
The left hand side is homogeneous of degree $n-i+1$ in $\alpha$
and homogeneous of degree $n-1$ in $\beta$, so we conclude that
the right hand side has to be of the form
$$q(\alpha,\beta)^{n-i+1}R_1(q(\beta,l),q(\beta,e),q(\beta)),$$
where the polynomial $R_1$ has to be homogeneous of degree
$i-2$ in $\beta$, hence $R_1$ has to be of weighted degree $i-2$ in the three variables
$q(\beta,l),\,q(\beta,e)$ of degree $1$ and $q(\beta)$ of degree $2$.
In conclusion, the space of coisotropic classes
$z$ which can be written as polynomials in $c,\,l$ and $e$ is the kernel of
the linear map $P\mapsto R_1$ which sends the space
$W_{2,1,1}^i$ of polynomials of weighted degree
$i$ in the variables $c,\,l,\,e$ to the space of polynomials $W_{2,1,1}^{i-2}$ of weighted degree
$i-2$ in $q(\beta),\,q(\beta,l),\,q(\beta,e)$. Thus
its kernel has dimension $\geq  {\rm dim}\,W_{2,1,1}^i-{\rm dim}\,W_{2,1,1}^{i-2}=i+1$.

The argument when $\rho(X)=1$ is exactly the same, except that there
is only one variable $l$ instead of the two variables $l,\,e$. We conclude
as before that the space of coisotropic classes
$z$ which can be written as polynomials in $c,\,l$
has dimension $\geq {\rm dim}\,W_{2,1}^i-{\rm dim}\,W_{2,1}^{i-2}=1$.

\vspace{0.5cm}

(iv) Writing $z$ as a polynomial in $c_{tr}$ and $l$, the non-vanishing
of the $l^i$ coefficient is equivalent
to the fact that $z$ is not of the form
  $c_{tr}\cup z'$, for some Hodge class $z'$ of degree $2i-4$.  So assume
  that $z=c_{tr}\cup z'$ is coisotropic. One then has
  \begin{eqnarray}\label{eqavion}\sigma^{n-i+1}\cup c_{tr}\cup z'=0
  \,\,{\rm in }\,\,H^{2n+2}(X,\mathbb{C}).
  \end{eqnarray}
  We now have the following lemma
  \begin{lemm}\label{propavion} The cup-product map
$$\cup c_{tr}: H^{2n-2}(X,\mathbb{C})\rightarrow H^{2n+2}(X,\mathbb{C})$$
is injective on the subspace
$ S^{n-1}H^2(X,\mathbb{C})\subset H^{2n-2}(X,\mathbb{C})$.
\end{lemm}
\begin{proof}
We first claim that
the cup-product map
$$\cup c: H^{2n-2}(X,\mathbb{C})\rightarrow H^{2n+2}(X,\mathbb{C})$$
is injective on the subspace
$ S^{n-1}H^2(X,\mathbb{C})\subset H^{2n-2}(X,\mathbb{C})$.
(Note that this is equivalent to saying that it induces
an isomorphism between $S^{n-1}H^2(X,\mathbb{C})\subset H^{2n-2}(X,\mathbb{C}$
and the subspace of $H^{2n+2}(X,\mathbb{C})$ which is the image
of $S^{n+1}H^2(X,\mathbb{C})$.)  The claim follows from the fact that
the kernel of the map
$$S^{n+1}H^2(X,\mathbb{C})\rightarrow H^{2n+2}(X,\mathbb{C})$$
is according to \cite{verb} equal to
$<QH^2(X,\mathbb{C})>_{n+1}$. On the other hand, representation theory of the orthogonal group
shows that the natural map
$$\oplus_k <QH^2(X,\mathbb{C})>_{n+1-2k}\rightarrow S^{n+1}H^2(X,\mathbb{C})$$
which is the multiplication by $c^k$ on $<QH^2(X,\mathbb{C})>_{n+1-2k}$,
is an isomorphism. Writing a similar decomposition
for $S^{n-1}H^2(X,\mathbb{C})$
 makes clear that the image of the multiplication
by $c$ from $S^{n-1}H^2(X,\mathbb{C})$ to $S^{n+1}H^2(X,\mathbb{C})$
 and the subspace $<QH^2(X,\mathbb{C})>_{n+1}$ have their intersection reduced to $0$,
 This proves the claim.

 To conclude the proof, we observe that the intersection pairing $q$ restricted
 to
 $H^2(X,\mathbb{C})_{tr}$ is nondegenerate, and that $c_{tr}\in S^2H^2(X,\mathbb{C})_{tr}$
 is the analogue of the class $c$ for $(H^2(X,\mathbb{C})_{tr},q)$. So Lemma
 \ref{propavion} applies to the
 multiplication maps
 \begin{eqnarray}
 \label{isoeqavion}c_{tr}: S^{k-1}H^2(X,\mathbb{C})_{tr}\rightarrow S^{k+1}H^2(X,\mathbb{C})_{tr}/<QH^2(X,\mathbb{C})_{tr}>_{k+1},
 \end{eqnarray}
 showing they are all injective.

To conclude, we observe that as
$H^2(X,\mathbb{C})=H^2(X,\mathbb{C})_{tr}\oplus \mathbb{C}l$, the space
$S^{n-1}H^2(X,\mathbb{C})$ decomposes as
\begin{eqnarray}\label{eqavion3}S^{n-1}H^2(X,\mathbb{C})=\oplus_{k=0}^{n-1} l^kS^{n-1-k}H^2(X,\mathbb{C})_{tr},
 \end{eqnarray}
 and similarly
 \begin{eqnarray}\label{eqavion4}S^{n+1}H^2(X,\mathbb{C})=\oplus_{k=0}^{n+1} l^kS^{n+1-k}H^2(X,\mathbb{C})_{tr}.
 \end{eqnarray}
 The three spaces $S^{n-1}H^2(X,\mathbb{C})$, $S^{n+1}H^2(X,\mathbb{C})$
 and $ <QH^2(X,\mathbb{C})>_{n+1}\subset S^{n+1}H^2(X,\mathbb{C})$
 are filtered
 by the respective subspaces $$l^k S^{n-1-k}H^2(X,\mathbb{C}),\,\,
 l^k S^{n+1-k}H^2(X,\mathbb{C}),\,\,l^k S^{n+1-k}H^2(X,\mathbb{C})\cap <QH^2(X,\mathbb{C})>_{n+1},$$ and denoting by
 $L^\cdot$ these filtrations, it is easy to check that
 \begin{eqnarray}\label{eqavion5}
 Gr_L^kS^{n-1}H^2(X,\mathbb{C})\cong S^{n-1-k}H^2(X,\mathbb{C})_{tr},\\
 \nonumber
 Gr_L^kS^{n+1}H^2(X,\mathbb{C})\cong S^{n+1-k}H^2(X,\mathbb{C})_{tr},\\
 \nonumber
  Gr_L^k<QH^2(X,\mathbb{C})>_{n+1}\cong <QH^2(X,\mathbb{C})_{tr}>_{n+1-k}.
 \end{eqnarray}
 The multiplication (or cup-product) map by $c_{tr}$ is compatible with the filtrations
 $L^\cdot$ (where we also use the induced filtration on the quotient
 $S^{n+1}H^2(X,\mathbb{C})/<QH^2(X,\mathbb{C})>_{n+1}\subset H^{2n+2}(X,\mathbb{C})$)
 and looking at the identifications (\ref{eqavion5}), we see that it induces between
  the graded pieces
 $Gr_L^kS^{n-1}H^2(X,\mathbb{C})$, $Gr_L^kS^{n+1}H^2(X,\mathbb{C})/<QH^2(X,\mathbb{C})>_{n+1}$
 the isomorphisms
 $$c_{tr}:S^{n-1-k}H^2(X,\mathbb{C})_{tr}\cong S^{n+1-k}H^2(X,\mathbb{C})_{tr}/<QH^2(X,\mathbb{C})>_{n+1-k}$$
 of (\ref{isoeqavion}).

 As the multiplication by $c_{tr}$ induces an isomorphisms on each graded piece, it is an isomorphism.

\end{proof}
Using Lemma \ref{propavion}, (\ref{eqavion}) implies that
\begin{eqnarray}\label{eqavion2}\sigma^{n-i+1}\cup z'=0\,\,{\rm in}\,\, H^{2n-2}(X,\mathbb{C}).
 \end{eqnarray}
 On the other hand, we know by
\cite{verb} that the natural map
$S^*H^2(X,\mathbb{C})\rightarrow H^{*}(X,\mathbb{C})$ is injective in degree
$*\leq n$, so that (\ref{eqavion2}) implies $z'=0$, hence $z=0$. So (iv) is proved.

\end{proof}

\section{The case of $Hilb^n(K3)$ \label{sec1}}
Let $S$ be a smooth surface and let $x_1,\ldots, x_i\in S$
be  $i$ different points.
We then  get
rational maps
$$S^{[n-i]}\dashrightarrow S^{[n]},$$
$$ z\mapsto \{x_1,\ldots,x_i\}\cup z,$$
which is well-defined at the points $z\in S^{[n-i]}$ parameterizing
subschemes of $S$ disjoint from the $x_l$'s.
These maps induce morphisms ${\rm CH}_0(S^{[i]})\rightarrow {\rm CH}_0(S^{[n]})$.
\begin{rema} {\rm If we work with the symmetric products $S^{(k)}$ instead of
 the Hlbert schemes $S^{[k]}$, the indeterminacies of these maps
do not appear anymore. Furthermore, as the Hilbert-Chow map has
rationally connected fibers, ${\rm CH}_0(S^{[k]})={\rm
CH}_0(S^{(k)})$. Finally, the fact that $S^{(k)}$ is a quotient
allows to work with correspondences and Chow groups of $S^{(k)}$
despite  their singularities (see \cite{deca}). Because of this, for
the computations below, we will work  with the symmetric products,
and this allows us to take $x_1=\ldots=x_i$. }
\end{rema}
According to the remark above, we choose now a point $o\in S$
and do $x_1=\ldots=x_i=o$. We denote by $(io):S^{(n-i)}\rightarrow S^{(n)}$ the
map $z\mapsto io+z$.
 The elementary
computations
leading to the following result already appear in \cite{bloch}, \cite{voisinsym},   \cite{moopo},
\cite{shenvial}.
In the following statement, we denote by
$\Sigma_{k-1,k}\subset S^{(k-1)}
\times S^{(k)}$ the correspondence
$$\Sigma_{k-1,k}=\{(z,z')\in S^{(k-1)}
\times S^{(k)},\,z\leq z'\}.$$
These correspondences induce morphisms
$$\Sigma_{k-1,k}^*:{\rm CH}_0(S^{(k)})\rightarrow {\rm CH}_0(S^{(k-1)})$$
where it is prudent here to take Chow groups with $\mathbb{Q}$-coefficients, due to the singularities of $S^{(k)}$.
\begin{prop}\label{prodecompchow} (i) The natural (but depending on $o$) decreasing filtration $N_i$ defined
on ${\rm CH}_0(S^{[n]})={\rm CH}_0(S^{(n)})$ by
\begin{eqnarray}\label{eqnatfilt} N_i{\rm CH}_0(S^{(n)}):={\rm Im}\,((io)_*:{\rm CH}_0(S^{(n-i)})
\rightarrow {\rm CH}_0(S^{(n)}))
\end{eqnarray}
induces a splitting
\begin{eqnarray}\label{eqnatdec}{\rm CH}_0(S^{(n)})=\oplus_{0\leq i\leq n}(io)_*({\rm CH}_0(S^{(n-i)})^0),
\end{eqnarray}
where $${\rm CH}_0(S^{(n-i)})^0:={\rm Ker}\,(\Sigma_{n-i-1,n-i}^*:{\rm CH}_0(S^{(n-i)})\rightarrow
{\rm CH}_0(S^{(n-i-1)})).$$
(ii) This splitting also induces a decomposition of each $N_k$:

\begin{eqnarray}\label{eqnatdecNi}N_k{\rm CH}_0(S^{(n)})=\oplus_{i\geq k}(io)_*({\rm CH}_0(S^{(n-i)})^0),
\end{eqnarray}
(iii) If $b_1(S)=0$, this decomposition gives a decomposition of the Bloch-Beilinson filtration, in the sense
that it is given by projectors $P_i$ acting on ${\rm CH}_0(S^{(n)})$,
the corresponding action on holomorphic forms being given
by
$$P_i^*=Id\,\,{\rm on}\,\,H^{2n-2i,0}(S^{(n)}),\,\,P_i^*=0\,\,{\rm on}\,\,H^{j,0}(S^{(n)}),\,j\not=2n-2i.$$
\end{prop}

\begin{proof}
(i)  The proof
is based on the
following easy formula (see \cite{voisinsym}):
\begin{eqnarray}\label{eqnsigman-1n} \Sigma_{n-1,n}^*\circ {(o)}_*=Id + (o)_*\circ \Sigma_{n-2,n-1}^*:{\rm CH}_0(S^{(n-1)})\rightarrow {\rm CH}_0(S^{(n-1)}).
\end{eqnarray}
Note that in this formula, the first $(o)_*$ belongs
to ${\rm Hom}\,({\rm CH}_0(S^{(n-1)}),{\rm CH}_0(S^{(n)}))$ and the second one
belongs to ${\rm Hom}\,({\rm CH}_0(S^{(n-2)}),{\rm CH}_0(S^{(n-1)}))$.
We deduce from this formula that if
$z=(o)_*(z')\in {\rm Ker}\,\Sigma_{n-1,n}^*\cap {\rm Im}\,(o)_*$, then
\begin{eqnarray}\label{eqnsigman-2n}z'=-(o)_*\circ \Sigma_{n-2,n-1}^*z'.
\end{eqnarray} Thus $z'\in{\rm Im}\,(o)_*$.
Applying
$\Sigma_{n-2,n-1}^*$ to both sides of  equality (\ref{eqnsigman-2n})
and formula (\ref{eqnsigman-1n}), one gets
$$\Sigma_{n-2,n-1}^*z'=-\Sigma_{n-2,n-1}^*z'-(o)_*(\Sigma_{n-3,n-2}^*\circ \Sigma_{n-2,n-1}^*z'),$$
and applying $(o)_*$ again,
one gets
$$ z'=-(o)_*\circ \Sigma_{n-2,n-1}^*z'=\frac{1}{2}(2o)_*(\Sigma_{n-3,n-2}^*\circ \Sigma_{n-2,n-1}^*z'), $$
so
that in fact $z'\in {\rm Im}\,(2o)_*$. Iterating this argument, we finally conclude that
$z=0$.

On the other hand, any
cycle in $S^{(n)}$ is the image of a  $0$-cycle in $S^n$ and each point of
$S^n$ can be written (as a $0$-cycle of $S^n$) as
\begin{eqnarray}
\label{eqpourpointinSn}(x_1,\ldots,x_n)=pr_1^*(x_1-o)\cdot \ldots\cdot pr_n^*(x_n-o)+z'
\end{eqnarray}
where $z'=\sum_in_iz'_i$ is a cycle of $S^n$ supported on points
$z'_i=(z'_{i,1},\ldots,z'_{i,n})$ having at least one
factor equal to $o$. Projecting to $S^{(n)}$, we conclude
that every $0$-cycle in $S^{(n)}$ is the sum of
a $0$-cycle in ${\rm Im}\,(o)_*$, and of a $0$-cycle
$(x_1-o)*\ldots*(x_n-o)$, where the $*$-product used here is the external product
 appearing in (\ref{eqpourpointinSn}) followed by the projection to
 $S^{(n)}$.
 It is immediate to check that $(x_1-o)*\ldots*(x_n-o)$ is annihilated by
 $\Sigma_{n-1,n}^*$, and thus we get the existence of a decomposition
 \begin{eqnarray}
\label{eqpourpreuvepropdec}{\rm CH}_0(S^{(n)})={\rm Im}\,(o)_*+{\rm Ker}\,\Sigma_{n-1,n}^*,
\end{eqnarray}
 this decomposition being in fact a direct sum decomposition by the previous argument.
 Using the decomposition (\ref{eqpourpreuvepropdec}), (i) is proved by induction.

(ii) This follows directly from the definition of $N_k$ and the decomposition
(\ref{eqnatdec}) applied to $S^{(n-k) }$.

(iii) We work by induction on $n$. We observe that the proof of
(i) shows that $${\rm Ker}\,(\Sigma_{n-1,n}^*:{\rm CH}_0(S^{(n)})\rightarrow
{\rm CH}_0(S^{(n-1)}))$$
identifies for $n>0$ to the image
of the map
$$*_n: {\rm Sym}^n({\rm CH}_0(S)_0)\rightarrow {\rm CH}_0(S^{(n)})$$
induced by the $*$-product, where ${\rm CH}_0(S)_0$ denotes the
group of $0$-cycles of degree $0$. Furthermore, we have a projector
$P_n$ from ${\rm CH}_0(S^{(n)})$ to ${\rm Im}\,*_n$, which to
$x_1+\ldots +x_n$ associates the cycle $(x_1-o)*\ldots*(x_n-o)$.
This projector annihilates ${\rm Im}\,(o_*:{\rm CH}_0(S^{(n-1)})
\rightarrow {\rm CH}_0(S^{(n)}))$ and it is thus the projector on
the summand ${\rm Ker}\,(\Sigma_{n-1,n}^*)$ in the decomposition
(\ref{eqpourpreuvepropdec}). Finally, we note that $P_n$ acts as the
identity on the space $H^{2n,0}(S^{[n]})={\rm Sym}^n H^{2,0}(S)$ and
as $0$ on the spaces of holomorphic forms of even degree $<2n$.

\end{proof}
\begin{rema}\label{remanoodd}{\rm Under the assumption made in (iii), $S^{(n)}$ (or rather
$S^{[n]}$) has no nonzero odd degree holomorphic form. For this
reason, the Bloch-Beilinson filtration on ${\rm CH}_0(S^{[n]})$
jumps only in even degree, that is $F^{2j}_{BB}{\rm
CH}_0(S^{[n]})=F^{2j-1}_{BB}{\rm CH}_0(S^{[n]})$. It is thus  more
natural in this case and also in the case of hyper-K\"ahler
varieties to work with the filtration ${F'} ^i_{BB}{\rm
CH}_0(S^{[n]}):=F^{2i}_{BB}{\rm CH}_0(S^{[n]})$, whose graded pieces
are governed by the $(2i,0)$-forms.}
\end{rema}
\begin{rema} \label{remafiltBB} {\rm One may prefer to use Proposition
\ref{prodecompchow}, (ii) as giving a construction of the
Bloch-Beilinson filtration ${F'} ^i_{BB}$ on ${\rm CH}_0(S^{[n]})$,
where $S$ is any surface with $q=0$. One remark is that, putting
$${F'}^i_{BB}{\rm CH}_0(S^{[n]})={F'}^i_{BB}{\rm CH}_0(S^{(n)}):=\oplus_{k\leq n-i}(ko)_*({\rm CH}_0(S^{(n-k)})^0)$$
the filtration $F'_{BB}$ does not depend on the choice of the point $o$.
}

\end{rema}
With this notation,   Proposition \ref{prodecompchow}, (ii) implies that
the fitrations $N$ and $F'_{BB}$ are opposite, in the sense
that the natural composite map
$$N_i{\rm CH}_0(S^{[n]})\rightarrow {\rm CH}_0(S^{[n]})\rightarrow {\rm CH}_0(S^{[n]})/{F'}^{n-i+1}_{BB}{\rm CH}_0(S^{[n]})$$
is an isomorphism. Let now $S$ be a projective $K3$ surface and let
$o_S\in{\rm CH}_0(S)$ be the canonical $0$-cycle of degree $1$ on
$S$ which is introduced in \cite{beauvoisin}. We choose for $o$ any
point representing the cycle $o_S$ and conclude that the filtration
$N_\cdot$ appearing in Proposition \ref{prodecompchow} is
canonically defined on ${\rm CH}_0(S^{[n]})$.

The following is obtained by reinterpreting Theorem 2.1 of \cite{voisinlaz}:
\begin{theo} \label{theofiltcomp} Let $X=S^{[n]}$. Then the filtration $N_\cdot$ introduced above and the filtration
$S_\cdot$ introduced in Definition \ref{defifilt} coincide on ${\rm CH}_0(X)$.
\end{theo}
\begin{proof} Let $x\in S_iX$. This means by definition
that there exists a subvariety $W_x\subset X$ which is of dimension
at least $i$, such that all points in $W_x$ are rationally
equivalent to $x$ in $X$. A fortiori, the  degree $n$ effective
$0$-cycles of  $S$ parameterized by $W_x$ are constant in ${\rm
CH}_0(S)$, since they are obtained by applying the universal
subvariety $\Sigma_n\subset S^{[n]}\times S$, seen as a
correspondence between $S^{[n]}=X$ and $S$, to the points of $W_x$.
We now apply the following result of \cite{voisinlaz} (Theorem 2.1
and Variant 2.4).
\begin{theo} Let $S$ be a projective $K3$ surface
and let $Z\subset S^{[n]}$ be a subvariety of dimension $i$, such that all the degree
$i$ effective $0$-cycles of $S$
parameterized by $Z$ are rationally equivalent in $S$.
Then for some constant cycle curve $C\subset S$, the image of $Z$ in $S^{(n)}$ intersects
$C^{(i)}+S^{(n-i)}\subset S^{(n)}$.

\end{theo}
We apply this result to $W_x$ and conclude that the image $\overline{W}_x$ of $W_x$
in $S^{(n)}$ intersects
$C^{(i)}+S^{(n-i)}\subset S^{(n)}$ in a point $z$. As all points in $W_x$ are rationally equivalent to
$x$ in $X$, the image $\overline{x}$ of $x$ in $S^{(n)}$ is rationally equivalent
 in $S^{(n)}$
to $z\in C^{(i)}+S^{(n-i)}$. But for any such point, which
is of the form
$z=z_1+z_2$ with $z_1$ effective  of degree  $i$ supported in $C$, $z_2\in S^{(n-i)}$, we
have $z_1=\sum_j x_j, \,x_j\in C\subset S$, hence $x_j$ is rationally equivalent
to $o_S$ in $S$, and thus $\sum_j x_j+z_2$ is rationally equivalent in $S^{(n)}$ to
the point $io_S+z_2$. We conclude that
$z$, hence also $\overline{x}$, is rationally equivalent  in $S^{(n)}$ to
a point in  the image of the map $io_S$,  so that its class in
${\rm CH}_0(S^{(n)})={\rm CH}_0(S^{[n]})={\rm CH}_0(X)$
belongs to
$S_i{\rm CH}_0(X)$.

The reverse inclusion is obvious since for any constant
cycle curve $C\subset S$, a point $z=io_S+z',\,z'\in S^{(n-i)}$ in ${\rm Im}\,(io_S:S^{(n-i)}\rightarrow
 S^{(n)})$ contains in its orbit the $i$-dimensional subvariety
 $C^{(i)}+z'$ which lifts to an $i$-dimensional subvariety
 of $X$, so that any lift of $z$ to $X$ belongs to $S_iX$.

\end{proof}

This result suggests  that the filtration $S_\cdot$, which is
defined for any  variety, could be  in the case of general  hyper-K\"ahler manifolds
the natural substitute
for the filtration $N_\cdot$ (which was defined only for
Hilbert schemes of surfaces).   We will formulate
this more explicitly in Section \ref{sec3}.

\section{Conjectures on the Chow groups of hyper-K\"ahler varieties\label{sec3}}

Let $X$ be a  hyper-K\"ahler projective manifold. We have the notion
of coisotropic class introduced in Definition \ref{defcoisoclass}.
We proved in  Section \ref{sec2} that the class of a codimension $i$
subvariety $Z$ of $X$ contained in $S_iX$ is coisotropic (see
Theorem \ref{propcoiso} which proves a stronger statement). The
computations made in Section \ref{sec11} (see Theorem
\ref{theoisoclass}) show that nonzero  coisotropic Hodge classes
always exist (and even nonzero  coisotropic algebraic classes,
assuming the algebraicity of the class $c$, see Remark
\ref{rema28dec}). Let us state the following more precise version of
Conjecture \ref{conjmain}:
\begin{conj}\label{conjgen} For any $X$ as above and any $i\leq n$, the space of coisotropic classes
of degree $2i$ is generated over
$\mathbb{Q}$ by classes of codimension $i$ subvarieties $Z$ of $X$ contained
in $S_iX$.
\end{conj}
\begin{rema}{\rm It might even be true, as suggested by the work of Charles and Pacienza
\cite{chapa}, that the space of coisotropic classes
of degree $2i$ is generated over
$\mathbb{Q}$ by classes of codimension $i$ subvarieties $Z$ of $X$ which are
swept-out by $i$-dimensional rationally connected varieties.}
\end{rema}
\begin{rema}{\rm There are three different problems hidden in
Conjecture \ref{conjgen}:

1) The Hodge conjecture: one has to prove that coisotropic Hodge
classes are algebraic.

2) The existence problem for coisotropic subvarieties: one has to
prove that the classes of coisotropic subvarieties  generate the space
of coisotropic classes. As the previous one, this problem does not
appear for divisors, as all divisors are coisotropic.

3) The existence problem for algebraically coisotropic subvarieties,
and even algebraically coisotropic subvarieties obtained as
codimension $i$ components of $S_iX$. The last problem is unsolved
even for divisors, but there are progresses in this case (for
example the problem is solved by Charles-Pacienza \cite{chapa} if $X$ is a
deformation of ${\rm Hilb}(K3)$).}
\end{rema}
 Let us prove the following conditional result.
\begin{theo}\label{theocondcond} Assume $X$ satisfies Conjecture \ref{conjgen}. Then
$$S_i{\rm CH}_0(X)={\rm Im}\,(z:{\rm CH}_i(X)\rightarrow {\rm CH}_0(X)),$$
for any  adequate combination $z=\sum_j n_j z_j\in {\rm CH}^i(X)$ of
classes of subvarieties $Z_j\subset S_iX$ of codimension $i$ in $X$, such that the  class
$[z]$ is a nonzero polynomial in $c$ and an ample divisor class
$l\in {\rm NS}(X)$.
\end{theo}

\begin{proof} Note first of all that such a $z$ exists if Conjecture \ref{conjgen} holds, since
Theorem \ref{theoisoclass}  shows the existence of a nonzero
coisotropic class which is a polynomial in $c$ and any given ample
class $l$.

 Next, for any such $z$, the inclusion ${\rm Im}\,(z:{\rm CH}_i(X)\rightarrow {\rm CH}_0(X))\subset
 S_i{\rm CH}_0(X)$ is obvious, since ${\rm Supp}\,z\subset S_iX$ and by definition
 $S_i{\rm CH}_0(X)$ is generated by the classes of the points in $S_iX$.

Let us prove the inclusion $\subset$. Let $x\in S_iX$. By assumption, there exists
 an $i$-dimensional subvariety
 $W_x\subset X$ all of whose points are rationally equivalent to $x$ in $X$.
We claim that ${\rm deg}(W_x\cdot z)\not=0$. Assuming the claim,
 $W_x\cdot z\in {\rm CH}_0(X)$ is a $0$-cycle of degree different from $0$
 supported on $W_x$, hence a nonzero multiple of $x\in {\rm CH}_0(X)$.
 This gives us the  inclusion
 $$S_i{\rm CH}_0(X)\subset {\rm Im}\,(z:{\rm CH}_i(X)\rightarrow {\rm CH}_0(X))$$
 since by definition $S_i{\rm CH}_0(X)$ is generated by the classes of such points $x$.

 Let us prove the claim. We use the fact (see Corollary \ref{coroiso}) that a constant cycle subvariety
 $W$
 is isotropic, that is $\sigma_{|W_{reg}}=0$ as a form.
 Equivalently, the pull-back of $\sigma$ to a desingularization
 $\widetilde{W}$ of $W$ has a vanishing class in
 $H^2(\widetilde{W},\mathbb{C})$. The Hodge structure on $H^2(X,\mathbb{Q})_{tr}$ being simple
 because $h^{2,0}(X)=1$, it follows that the restriction map
 $H^2(X,\mathbb{Q})_{tr}\rightarrow H^2(\widetilde{W},\mathbb{Q})$ vanishes identically.
 This implies that the class $c_{tr}$, which by construction
 belongs to $S^2H^2(X,\mathbb{Q})_{tr}\subset H^4(X,\mathbb{Q})$, vanishes in
$ H^4(\widetilde{W},\mathbb{Q})$. As we know by Theorem
\ref{theoisoclass}, (iv) that $[z]=\lambda l^i+c_{tr}l^{i-2}+...$,
with $\lambda\not=0$, we get that ${\rm deg}(W\cdot z)=\lambda{\rm
deg}(W\cdot l^i)\not=0$.

\end{proof}

This result suggests that Beauville's conjectural splitting of the
Bloch-Beilinson filtration could be obtained by considering the
action of the classes in ${\rm CH}(X)$ of codimension $i$
subvarieties of $X$ contained in $S_iX$. More precisely, we  would
like to impose the following rule: Let $C_{2n-i}(X)\subset {\rm CH}^i(X)$
be the $\mathbb{Q}$-vector space generated by codimension $i$
components of $S_iX$.

\vspace{0.5cm}

{\it (*) The subspace $C_{2n-i}(X)\subset {\rm CH}^i(X)$ is
contained in the $0$-th piece ${\rm CH}^i(X)_0$ of the Beauville
conjectural splitting.}

\vspace{0.5cm}

 For this to be compatible with multiplicativity (and the axiom that
 $F^1_{BB}{\rm CH}={\rm CH}_{hom}$), one needs to prove the following concrete
conjecture:
\begin{conj}\label{conjclasscoiso} Let $X$ be hyper-K\"ahler of dimension $2n$.
Then the cycle class map is injective on the subalgebra of ${\rm
CH}^*(X)$ generated by $\oplus_iC_{2n-i}(X)$.
\end{conj}
Restricting to the $\mathbb{Q}$ vector subspace  $Z_{2n-i}(X)$  of
${\rm CH}^i(X)$ generated by codimension $i$ subvarieties contained
in $S_iX$ for any given $i$, we have to prove in particular:
\begin{conj}\label{conjvar} For any $i$ such that
$0\leq i\leq n$, the cycle class map is injective on the
$\mathbb{Q}$-vector subspace $C_{2n-i}(X)$ of ${\rm CH}^i(X)$.

In particular, for $i=n$, $2n={\rm dim}\,X$, the cycle class map is
injective on the subspace $C_n(X)$ generated by classes in ${\rm
CH}(X)$ of constant cycle Lagrangian subvarieties of $X$.
\end{conj}
An evidence for this conjecture is provided by Proposition
\ref{propconjvar}, which establishes it for the Fano variety of lines of a cubic fourfold.

A last conjecture suggested by the results in Section \ref{sec1} concerns
the possibility of constructing the conjectural Beauville decomposition from
the filtration $S_\cdot$ studied in the previous section, at least on some part of
${\rm CH}(X)$. Here we assume of course the existence of the Bloch-Beilinson filtration.
First of all,  let us consider the case of ${\rm CH}_0$.

\begin{conj} \label{conjSiop}Let $X$ be hyper-K\"ahler of dimension $2n$.
For any integer $i$ such that
$0\leq i\leq n$, the filtration $S_\cdot$ is opposite to the filtration
$F'_{BB}$ in the sense that
$$S_ i{\rm CH}_0(X)\cong {\rm CH}_0(X)/{F'}_{BB}^{n-i+1}{\rm CH}_0(X)=0.$$
\end{conj}
The main evidences for this conjecture are the cases of $S^{[n]}$,
where $S$ is a $K3$ surface (see Sections \ref{sec1} and \ref{sec4})
and of the Fano variety of lines of a cubic fourfold (see
Proposition \ref{procascubic}), for which we already have candidates for the Bloch-Beilinson
filtrations.
\begin{rema}\label{rema1jan}{\rm In the case
$i=n$, Conjecture \ref{conjSiop} takes a more concrete form which does not assume  the existence of the Bloch-Beilinson filtration. Indeed,
we then have ${F'}_{BB}^{n-i+1}{\rm CH}_0(X)=F^2_{BB}{\rm
CH}_0(X)={\rm CH}_0(X)_{hom}$ and we are considering $0$-cycles
supported on constant cycles Lagrangian subvarieties. The conjecture
is that they are rationally equivalent to $0$ if and only if they
are of degree $0$.}
\end{rema}
The following observation illustrates the importance of Conjecture \ref{conjmain}
for our constructions:
 \begin{lemm}\label{lesurjtoopp}
 Assuming Conjecture \ref{conjmain},
the map
\begin{eqnarray}\label{flechenat}
S_ i{\rm CH}_0(X)\rightarrow {\rm CH}_0(X)/ {F'}_{BB}^{n-i+1}{\rm CH}_0(X)
\end{eqnarray}
is surjective.
\end{lemm}
\begin{proof}
First we claim that  for any component
$Z$ of codimension $i$ of $S_iX$, the pull-back map
$$H^0(X,\Omega_X^l)\rightarrow H^0(\widetilde{Z},\Omega_{\widetilde{Z}}^l)$$
is injective for $l\leq 2(n-i)$, where $\widetilde{Z}$ is a
desingularization of $Z$. Indeed, the space $H^0(X,\Omega_X^l)$ is
equal to $0$ for odd $l$ and is generated by the form $\sigma^{l'}$
for $l=2l'$. The form $\sigma$ being everywhere nondegenerate, the
rank of $\sigma_{\mid {Z}_{reg}}$ is at least $2n-2i$, which implies
that $\sigma^{n-i}$ does not vanish in
$H^0(\widetilde{Z},\Omega_{\widetilde{Z}}^{n-i})$.

By the general axioms concerning
the Bloch-Beilinson filtration, the claim above  guarantees  the surjectivity
of the map (\ref{flechenat}).
\end{proof}
Conjecture \ref{conjSiop} thus concerns the injectivity of this map. One case would be also  an easy consequence
of Conjecture \ref{conjgen}:
\begin{lemm}\label{lechapa} (i) (Charles-Pacienza \cite{chapa}) Conjecture \ref{conjSiop} holds for constant Lagrangian surfaces
in very general algebraic deformations of ${\rm Hilb}^2(S)$.

(ii) More generally,  Conjecture \ref{conjSiop} holds for $i=n$ if
$X$ contains a constant Lagrangian subvariety which is connected and
of class $\lambda_nl^n+\lambda_{n-2}c_{tr}l^{n-2}+\ldots$ for some
ample class $l\in {\rm NS}(X)$.
\end{lemm}
\begin{proof} By Remark \ref{rema1jan}, what we have to prove when $i=n$
is the equality
 $$S_n{\rm CH}_0(X)\cap {\rm CH}_0(X)_{hom}=0.$$
  In
both cases (i) and (ii), we have the existence of a {\it connected}
Lagrangian constant cycle subvariety $W\subset X$ of class
$w=\lambda_nl^n+\lambda_{n-2}c_{tr}l^{n-2}+\ldots$, where $l$ is an
ample divisor class on $X$ and $\lambda_n\not=0$. (In the case (i),
this is because all Lagrangian surfaces have their class
proportional to $\lambda l^2+c_{tr}$, and in case (ii) this follows
from  our assumptions, using  Theorem \ref{theoisoclass} (iv)). The
same argument as in the proof of Theorem \ref{theocondcond} then
shows that for any Lagrangian constant cycle subvariety
$\Gamma\subset X$, one has ${\rm deg}(\Gamma\cdot W)\not=0$. It
follows that $\Gamma\cap W\not=\emptyset$, and thus that any point
of $\Gamma$ is rationally equivalent in $X$ to any point of $W$.
Thus $S_n{\rm CH}_0(X)=\mathbb{Q}$, and $S_n{\rm CH}_0(X)\cap {\rm
CH}_0(X)_{hom}=0$.

\end{proof}
Another small evidence for Conjecture \ref{conjSiop} is provided by
the following result: If $z$ is the class of any subvariety of $X$
contained in $S_iX$ and of codimension $i$, and $\Gamma\in {\rm
CH}^{2n-i}(X)$ is any cycle, $z\cdot \Gamma$ belongs to $S_i{\rm
CH}_0(X)$. Hence if $\Gamma \in F_{BB}^{2n-2i+1}{\rm CH}^{2n-i}(X)$
we have $z\cdot\Gamma\in F_{BB}^{2n-2i+1}{\rm CH}^{2n-i}(X)\cap
S_i{\rm CH}_0(X)$ and Conjecture \ref{conjSiop} then predicts  that
$$z\cdot \Gamma=0\,\,{\rm in}\,\,{\rm CH}_0(X).$$
This is in fact true, as shows the following result (which assumes  the existence of the Bloch-Beilinson filtration):
\begin{prop} Let $Z\subset X$ be a codimension $i$ subvariety contained in
$S_iX$. Then the intersection product
$$Z:F_{BB}^{2n-2i+1}{\rm CH}^{2n-i}(X)\rightarrow {\rm CH}_0(X)$$
vanishes identically.
\end{prop}
\begin{proof} We use Theorem
\ref{propcoiso}, (ii) which says that a desingularization
$\widetilde{Z}\stackrel{\tilde{i}}{\rightarrow } X$ of $Z$ admits a
fibration
$$p: \widetilde{Z}\rightarrow B$$
where ${\rm dim}\,B=2n-2i$, and the $i$-dimensional fibers of $p$
map via $\tilde{i}$ to constant cycle subvarieties of $X$. It
follows that the morphism
$$\tilde{i}_*:{\rm CH}_0(\widetilde{Z})\rightarrow {\rm CH}_0(X)$$
factors through ${\rm CH}_0(B)$. Indeed, let $B'\subset
\widetilde{Z}$ be generically finite of degree $N$ over $B$, and let
$p':B'\rightarrow B$ be the restriction of $p$. Then for any point
$z\in \widetilde{Z}$, one has
$$\tilde{i}_*(z)=\frac{1}{N}\tilde{i}_*({p'}^*(p_*z))\,\,{\rm in}\,\,{\rm CH}_0(X),$$
which provides the desired factorization.

Now, for $\Gamma\in F^{2n-2i+1}{\rm CH}^{2n-i}(X)$ one has
$Z\cdot\Gamma=\tilde{i}_*(\tilde{i}^*\Gamma)$. As
$\tilde{i}^*\Gamma\in F^{2n-2i+1}{\rm CH}_0(\widetilde{Z})$, one has
 $p_*(\tilde{i}^*\Gamma)\in F^{2n-2i+1}{\rm CH}_0(B)$ where the last space is equal to
 $\{0\}$ because
${\rm dim}\,B=2n-2i$. By the factorization above, this implies that
$\tilde{i}_*(\tilde{i}^*\Gamma)=0$ in ${\rm CH}_0(X)$.

\end{proof}
 Conjecture \ref{conjSiop}
would allow to construct the Beauville decomposition on ${\rm
CH}_0(X)$ as
$${\rm CH}_0(X)_{2k}= S_{n-k}{\rm CH}_0(X)\cap {F'}_{BB}^{k}{\rm CH}_0(X),\,\,{\rm CH}_0(X)_{2k+1}=0,$$
and we would have  equivalently
$$ S_i{\rm CH}_0(X)=\oplus_{ j\leq n-i} {\rm CH}_0(X)_{2j},$$
$${F'}_{BB}^i{\rm CH}_0(X)=\oplus_{j\geq i}{\rm CH}_0(X)_{2j}.$$
  If we try to extend this construction to other cycles,
having in mind that the Beauville decomposition is supposed to be
multiplicative and have the divisor classes in its $0$-th piece, the
following proposal seems to be compatible with Theorem
\ref{theocondcond} and the previous assignments:

From a decomposition \begin{eqnarray}
\label{eq*29}{\rm CH}^k(X)=\oplus_{0\leq j\leq k}{\rm CH}^k(X)_j
\end{eqnarray}
with
\begin{eqnarray}
\label{eq**29}F^i_{BB}{\rm CH}^k(X)=\oplus_{i\leq  j\leq k}{\rm CH}^k(X)_j,
\end{eqnarray} one can
construct another decreasing filtration
$T$  defined by
\begin{eqnarray}
\label{eq***29}
 T^i{\rm CH}^k(X)=\oplus_{0\leq j\leq k-i } {\rm CH}^k(X)_j.
\end{eqnarray}
One clearly has
$$ {\rm CH}^k(X)_j=T^{k-j}{\rm CH}^k(X)\cap F^j_{BB}{\rm CH}^k(X).$$
Conversely, a decreasing filtration $T$ which is opposite
to the filtration $F_{BB}$ in the sense
that
$$T^l{\rm CH}^k\cong {\rm CH}^k/ F_{BB}^{k+1-l}{\rm CH}^k$$
 leads to a decomposition (\ref{eq*29}) satisfying
(\ref{eq**29}) and (\ref{eq***29}).

For $0$-cycles, we put $T^{2i}{\rm CH}_0(X)=S_i{\rm CH}_0(X)$ and assuming
Conjecture \ref{conjSiop},  we have the desired decomposition.

We now propose the following assignment for the filtration $T$:

\vspace{0.5cm}

{\it (**)  Let $\Gamma\subset X$ be an $i$-dimensional constant
cycle subvariety. Then its class $\gamma\in {\rm CH}^{2n-i}(X)$
belongs to $T^i{\rm CH}(X)$.}

\vspace{0.5cm}

For a filtration $T$ satisfying the assignment above to be opposite to
the Bloch-Beilinson filtration, one needs the following:

\begin{conj}\label{conjpresquempty} Let
$C_i(X)\subset {\rm CH}_i(X)={\rm CH}^{2n-i}(X)$ be the
$\mathbb{Q}$-vector space generated by constant cycle subvarieties
of dimension $i$. Then the Bloch-Beilinson filtration $F_{BB}$
satisfies
\begin{eqnarray}\label{eq1jan}F_{BB}^{2n-2i+1}C_i(X)=0.
\end{eqnarray}
\end{conj}
Note that the general finiteness condition for the Bloch-Beilinson
filtration (see Section \ref{secBB}) is $$F^{2n-i+1}{\rm CH}^{2n-i}=0,$$  which is weaker than
(\ref{eq1jan}). Note also that Conjecture \ref{conjpresquempty}
generalizes Conjecture \ref{conjvar} in the case of Lagrangian
constant cycle subvarieties (i.e. the case $i=n$). Indeed, for
$i=n$, one has $2n-2i+1=1$, hence $F_{BB}^{2n-2i+1}{\rm
CH}^{2n-i}(X)={\rm CH}^{2n-i}(X)_{hom}$, and Conjecture
\ref{conjpresquempty} says that the cycle class map is injective on
the subspace of ${\rm CH}^n(X)$ generated by classes of Lagrangian
constant cycle subvarieties.
\begin{rema}{\rm The three conjectures \ref{conjfiltopp}, \ref{conjclasscoiso} and
\ref{conjpresquempty} can be unified as follows: Consider the
inclusion $S_iX\subset X$, and let $S'_iX\subset S_iX$ be the union
of the components of $S_iX$ which are of codimension $i$ in $X$.
According to Theorem \ref{propcoiso}, each component $Z$ of $S'_iX$
(or rather a birational model of it) admits a fibration
$Z\rightarrow B$ into $i$-dimensional constant cycles subvarieties,
with ${\rm dim}\,B=2n-2i$. For each such $Z$, we thus have a
correspondence between $B$ and $X$ given by the two maps
$$\tilde{i}:Z\rightarrow X,\,\,p:Z\rightarrow B,$$ and  we thus have three
maps, namely :

1) $\overline{\tilde{i}_*}:{\rm CH}_0(B)\rightarrow
{\rm CH}_0(X)$ factoring through $p_*:{\rm CH}_0(Z)\rightarrow{\rm
CH}_0(B)$ the natural map $\tilde{i}:{\rm CH}_0(Z)\rightarrow {\rm
CH}_0(X)$ using the fact that the fibers of $p$ map via $\tilde{i}$
to constant cycles subvarieties of $X$;

 2)
$$Z_*=\tilde{i}_*\circ p^*:{\rm CH}_0(B)\rightarrow {\rm CH}_i(X),$$
whose image is contained in the subgroup $C_i(X)$ generated by
classes of constant cycles subvarieties of dimension $i$;

3)
 $Z_*=\tilde{i}_*\circ p^*:{\rm CH}^0(B)\rightarrow {\rm CH}^i(X)$,
whose image belongs to the subgroup $C_{2n-i}(X)$.

Observing that the Bloch-Beilinson filtration on ${\rm CH}(B)$ satisfies
$$F^{2n-2i+1}_{BB}{\rm CH}_0(B)=0,\,F^1{\rm CH}^0(B)=0,$$
for all $Z,\,B$'s as above and taking the disjoint union of all
components $Z_j\subset X$ of $S'_iX$, and of the corresponding
$B_j$'s, our conjectures can be unified and even strengthened saying
that the three maps above are strict for the Bloch-Beilinson
filtrations.}
\end{rema}
\begin{rema}\label{remaarevoir}{\rm There are remarkable relations between  these three
maps, namely assuming $B$ is connected, there are coefficients
$\mu$, $\nu_l$ with $\nu_l\not=0$ depending on an ample class $l$
such that
\begin{eqnarray} Z_*(\alpha)\cdot
Z_*(\gamma)=\mu\overline{\tilde{i}_*}(\alpha\cdot\gamma)\,\,{\rm
in}\,\,{\rm CH}_0(X),
\end{eqnarray}
for any $\alpha\in {\rm CH}^0(B),\,\gamma\in {\rm CH}_0(B)$, and
\begin{eqnarray}\label{eqpri1} l^i\cdot
Z_*(\gamma)=\nu_l\overline{\tilde{i}_*}(\gamma)\,\,{\rm
in}\,\,{\rm CH}_0(X),
\end{eqnarray}
for any $\gamma\in {\rm CH}_0(B)$.

Both relations follow immediately from the fact that the fibers of
$p$ are constant cycles subvarieties of $X$, and they just say that
that for any such fiber $Z_t$, the intersection $Z\cdot Z_t$, resp.
$l^i\cdot Z_t$ is proportional to $\overline{\tilde{i}_*}(t)$ in
${\rm CH}_0(X)$.
 } \end{rema}
 The relation (\ref{eqpri1}) provides a close link between  Conjectures \ref{conjSiop}
 and Conjecture \ref{conjpresquempty} (we assume here the existence of a Bloch-Beilinson filtration).
 \begin{lemm}
 \label{leprin}(i)  One has ${\rm Ker}\,Z_*\subset {\rm Ker}\, \overline{\tilde{i}_*}\subset{\rm CH}_0(B)$.

  (ii) Assuming Conjecture \ref{conjSiop}, for any codimension $i$
 component $Z\subset X$ of $S_iX$, and any $0$-cycle $\gamma\in {\rm CH}_0(B)$,
 one has  the implications
 $$Z_*\gamma\in F^{2n-2i+1}{\rm CH}_i(X)\Rightarrow \,\,\overline{\tilde{i}_*}(\gamma)=0\,\,{\rm
in}\,\,{\rm CH}_0(X).$$
 In particular, if furthermore one has equality in (i), then
 $$Z_*\gamma\in F^{2n-2i+1}{\rm CH}_i(X)\Rightarrow Z_*\gamma=0\,\,{\rm in}\,\,{\rm CH}_i(X),$$
 which is essentially Conjecture \ref{conjpresquempty}.
 \end{lemm}
 \begin{proof} Indeed, (i) is an obvious consequence of (\ref{eqpri1}).
 As for (ii), if $Z_*\gamma\in F^{2n-2i+1}{\rm CH}_i(X)$, then
 $\overline{\tilde{i}_*}(\gamma)\in F^{2n-2i+1}{\rm CH}_0(X)$ by (\ref{eqpri1}).
 Conjecture \ref{conjSiop} says now that
 $F^{2n-2i+1}{\rm CH}_0(X)\cap S_i{\rm CH}_0(X)=0$, hence that
  $\overline{\tilde{i}_*}({\rm CH}_0(B))\cap F^{2n-2i+1}{\rm CH}_0(X)=0$. This implies that
  $\overline{\tilde{i}_*}(\gamma)=0$.
 \end{proof}
 From a slightly different point of view, let us explain how Conjecture \ref{conjpresquempty}
would lead to multiplicativity statements for the associated
decomposition. First of all, let us observe the following results
along the same lines as Theorem \ref{theocondcond}:
\begin{lemm} \label{theofin}(i) Let $\Gamma\subset X$ be a constant cycle subvariety
of dimension $i$. Then for any $l\in {\rm NS}(X)={\rm Pic}(X)$,
$l^i\Gamma\in S_i{\rm CH}_0(X)=T^{2i}{\rm CH}_0(X)$.

(ii) Assuming Conjecture \ref{conjgen}, $S_i{\rm CH}_0(X)=T^{2i}{\rm
CH}_0(X)$ is generated by products $Z\cdot \Gamma$, where $Z$ is a
codimension $i$ subvariety of $X$ contained in $ S_iX$, and $\Gamma$
a constant cycle subvariety of $X$ of dimension $i$.
\end{lemm}
\begin{proof}
(i) Indeed, as  all points of $\Gamma$ belong to $S_iX$ by definition, so does the
$0$-cycle $l^i\Gamma$ which is supported on $\Gamma$.

(ii) Let $x\in S_iX$ and  let $\Gamma_x$ be a constant cycle
subvariety of dimension $i$. Then for any cycle $z\in {\rm CH}^i(X)$
such that ${\rm deg}\,(z\cdot \Gamma_x)\not=0$, $z\cdot \Gamma_x \in
{\rm CH}_0(X)$ is a nonzero multiple of the class of any point of
$X$ supported on $\Gamma_x$, hence of $x$. Assuming Conjecture
\ref{conjgen}, the same argument as in the proof of Theorem
\ref{theocondcond} shows that there is a combination $z\in C_i(X)$
of classes of codimension $i$ subvarieties of $X$ contained in
$S_iX$, such that ${\rm deg}\,(z\cdot \Gamma_x)\not=0$ and thus the
class of $x$ is a multiple in ${\rm CH}_0(X)$ of $z\cdot \Gamma$,
which shows that $S_i{\rm CH}_0(X)$ is generated by products $Z\cdot
\Gamma$, where $Z\in C_{2n-i}(X)$, and $\Gamma\in C_i(X)$, since by
definition $S_i{\rm CH}_0(X)$ is generated by the classes of  points
in $S_iX$. The other inclusion is obvious since any cycle $Z\cdot
\Gamma$ with $Z\subset S_iX$ of codimension $i$ in $X$ is supported
on $Z$, hence belongs to $S_i{\rm CH}_0(X)$.

\end{proof}
Let now $\Gamma$ be a constant cycle subvariety of dimension $i$.
Then according to (**), the cycle $\Gamma\in {\rm CH}^{2n-i}(X)$
should belong to $T^i{\rm CH}^{2n-i}(X)=\oplus_{j\leq 2n-2i} {\rm
CH}^{2n-i}(X)_j$. According to (*), the class $z$ of any subvariety
of $X$ contained in $S_iX$ and of codimension $i$ should be in ${\rm
CH}^i(X)_0$. By multiplicativity of the conjectural Beauville
decomposition, one should have \begin{eqnarray}
\label{eqlille}z\cdot \Gamma \in \oplus_{j\leq 2n-2i} {\rm
CH}_0(X)_j, \end{eqnarray} the right hand side being equal to
$T^{2i}{\rm CH}_0(X)=S_i{\rm CH}_0(X)$. Equation (\ref{eqlille}) is
in fact satisfied by the easy inclusion in Lemma \ref{theofin},
(ii), thus providing some evidence for the multiplicativity of the
decomposition we started to construct. Similarly, if we take for $z$
a degree $i$ polynomial in divisor classes on $X$, then $z$ should
belong to  ${\rm CH}^i(X)_0$ and thus we should have according to
(**) and multiplicativity
$$z\cdot \Gamma \in \oplus_{j\leq 2n-2i} {\rm CH}_0(X)_j=S_i{\rm CH}_0(X)$$
which is proved in Lemma \ref{theofin}, (i).

\section{Examples \label{sec4}}
The purpose of this section is to collect some examples providing evidences
for the conjectures proposed in this paper.
\subsection{The case of ${\rm Hilb}(K3)$}
We first examine Conjecture \ref{conjgen} concerning the existence
of many algebraic coisotropic subvarieties obtained as codimension
$i$ components of $S_iX$.

Let us consider a very general algebraic $K3$ surface $S$, so that
${\rm NS}(S)$ has rank $1$ and is  generated by the class of $L\in
{\rm Pic}\,S$, and let $X:=S^{[n]}$. The N\'eron-Severi group ${\rm
NS}(X)$ has then rank $2$, and is generated over $\mathbb{Q}$ by the
class $e$ of the exceptional divisor and the class $l$ of the
pull-back to $X$, via the Hilbert-Chow morphism
$$s: X=S^{[n]}\rightarrow S^{(n)}$$
 of the  divisor $C+S^{(n-1)}\subset S^{(n)}$ where $C\in |L|$.

 \vspace{0.5cm}

1) Obvious examples of constant cycles subvarieties of $X$ are
provided by the fibers of  $s$. It is indeed known that these fibers
are rationally connected, so that they are constant cycles
subvarieties. We know  that for each stratum $S^{(\mu)}_0\subset
S^{(n)}$ determined by multiplicities $\mu_1,\ldots,\mu_l$ such that
$\sum_i\mu_i=n$, the inverse image $s^{-1}(S^{(\mu)}_0)$ is of
codimension $i$ in $S^{[n]}$, and fibered into constant cycles
subvarieties of dimension $i$, namely the fibers of $s$ over points
$z\in S^{(\mu)}_0$. Here the notation is as follows: The number $i$
is equal to $n-l$, and $S^{(\mu)}_0$ is the locally closed stratum
determined by $\mu$, namely
$$S^{(\mu)}_0:=\{\sum_j\mu_j x_j,\,x_j\in S,\,x_j\not=x_k,\,j\not=k\}.$$
The Zariski  closure $E_\mu$ of $s^{-1}(S^{(\mu)}_0)$ in $S^{[n]}$ is thus an example
of a codimension $i$ algebraically coisotropic subvariety of $S^{[n]}$ fibered into
$i$-dimensional constant cycles subvarieties, as studied in Section \ref{sec2}.

 \vspace{0.5cm}

1bis) With the same notation as above, let $W\subset
S^{(\mu)}_0\subset S^{(n)}$ be a codimension $j$ subvariety which is
fibered by $j$-dimensional  subvarieties $Z_t$ of $S^{(\mu)}$ which
are constant cycles for $S^{(n)}$ in the sense that for each $t$, the natural map
$Z_t\rightarrow {\rm CH}_0(S^{(n)})$ is constant. Then the locally
closed varieties $s^{-1}(Z_t)\subset S^{[n]}$ are of dimension
$j+n-l(\mu)$ and they sweep-out the locally closed subvariety
$s^{-1}(W)$ which has codimension $n-l(\mu)+j$. Hence its Zariski
 closure in $X$ is algebraically coisotropic, fibered
into constant cycles subvarieties.

\vspace{0.5cm}

2) Starting from a constant cycle curve $C\subset S$, for example a rational curve, we can
also get codimension $i$ subvarieties of $S_iX$, by taking the
image of the rational map
$$C^{(i)}\times S^{[n-i]}\dashrightarrow S^{[n]},$$
$$(z,z')\mapsto z\cup z'.$$

 \vspace{0.5cm}

3) A more subtle example (which however does not work for all
possible pairs $(i,g)$, where $L^2=2g-2$) is given by applying the
Lazarsfeld construction \cite{laz} in any possible range for the
Brill-Noether theory of smooth curves in $|L|$. Let us describe this
construction in more detail. Let $C\in |L|$ be smooth, and let
$|D|\in G_n^1(C)$ be a base-point free pencil. By the main result of
\cite{laz}, such a $D$ exists if and only if $2(g-n+1)\leq g$, that
is $g\leq 2n-2$. Let $F$ be the rank $2$ vector bundle on $S$ which
is defined as the kernel of the (surjective) evaluation map
$$ H^0(D)\otimes \mathcal{O}_S\rightarrow \mathcal{O}_C(D)$$
and let $E:=F^*$. Then
$E$ fits in an exact sequence
$$0\rightarrow H^0(D)^*\otimes \mathcal{O}_S\rightarrow E\rightarrow K_C(-D)\rightarrow 0.$$
It follows that
\begin{eqnarray} h^0(S,E)=2+g-n+1,\,{\rm deg}\, c_2(E)=n,
\end{eqnarray}
so that $0$-sets of sections of $E$ provide constant cycles
subvarieties of $S^{[n]}$ of dimension $g+2-n$ (we observe here that
$E$ satisfies $h^0(S,E(-L))=0$, hence any nonzero section of $E$ has
a $0$-dimensional vanishing locus, so that the morphism
$\mathbb{P}(H^0(S,E))\rightarrow S^{[n]}$ is well-defined, obviously
non-constant and thus finite to one onto its image). Let us now
compute the dimension of the subvariety $Z\subset S^{[n}]$ we get by
letting $(C,D)$ deform in the space of pairs consisting of a curve
$C\in|L|$ and an effective divisor $D$ which is a $g_n^1$ on $C$.
The curve $C$ moves in the $g$-dimensional linear system $|L|$ and
$\mathcal{O}_C(D)$ moves in the codimension $2(g-n+1)$ subvariety of
the relative Picard variety ${\rm Pic}(\mathcal{C}/|L|)$ which has
dimension $2g$ (here $\mathcal{C}\rightarrow |L|$ is the universal
curve; we work in fact over the open set parameterizing smooth
curves, and we use Lazarsfeld's theorem \cite{laz} saying that the
dimensions are the expected ones). The choice of $D\in C^{(n)}$
instead of the line bundle $ \mathcal{O}_C(D)$ provides one more
parameter. This gives us a subvariety $W$ of $\mathcal{C}^{(n/B)}$
of dimension $2g+1-(2(g-n+1))=2n-1$. Finally, the fiber over a point
$z\in S^{[n]}$ of the surjective map $W\rightarrow Z$ obtained by
restricting to $W$ the natural morphism
$$\mathcal{C}^{(n/B)}\rightarrow S^{[n]}$$
  identifies to the set of curves
$C\in |L|$ containing the $0$-dimensional subscheme $z\subset S$,
and this is a projective subspace of $|L|\cong\mathbb{P}^g$ which is
of dimension $g-n+1$ since these $z$'s impose exactly $n-1$
conditions to $|L|$. We conclude that ${\rm
dim}\,Z=2n-1-(g-n+1)=3n-g-2$, and ${\rm codim}\,Z=g+2-n$. Hence
$Z\subset S_{g+2-n}X$ and has codimension $g+2-n$.
\begin{rema}{\rm In this example, the base of the coisotropic fibration of $Z$ is
birationally equivalent to a moduli space
of rank $2$ vector bundles on $S$, hence to a
possibly singular hyper-K\"ahler variety. This needs not to be the case in general.
For example, there is a uniruled divisor in the variety of lines $F$
of  a cubic fourfold  which is uniruled over a surface of general type, namely
the indeterminacy locus of the  rational self-map $\phi:F\dashrightarrow F$ constructed in
\cite{voisinfano} (see also Subsection \ref{secsubfanocubic}).
}
\end{rema}
4)  For $n=2$, $S$ admits a covering by a $1$-parameter family of
elliptic curves. Each such curve $E$ carries the
``Beauville-Voisin'' $0$-cycle $o_S$, that is contains a point $x$
that is rationally equivalent to $o_S$ in $S$,  and $2x$ moves in a
pencil in $E$. This way we get a $2$-dimensional orbit $\Sigma$ of
$o_S$ in $S^{[2]}$, which is a Lagrangian surface.

\vspace{0.5cm}

4bis) We can combine construction 4) and the sum map
$$\mu:S^{[2]}\times S^{[n-2]}\dashrightarrow S^{[n]}$$
to construct codimension $2$ subvarieties of $S^{[n]}$ contained in
$S_2S^{[n]}$ : Namely, let $\Sigma $ be as above, then for each
$z\in S^{[n-2]}$, the image $\mu(\Sigma\times z)$ is a surface in
$S^{[n]}$ all of whose points are rationally equivalent in $S^{[n]}$
hence $\mu(\Sigma\times S^{[n-2]})$ is contained in $S_2S^{[n]}$.

Let us now  prove the following result, which is
Conjecture \ref{conjgen} for degree $4$  coisotropic
classes on $S^{[n]}$, $S$ a very general algebraic $K3$ surface. Note that
the cases $n=2,\,3,\,n\geq4$  differ
 from the viewpoint of computing
 cohomology of degree $4$. In the case $n=2$,
 the degree $4$ cohomology (resp. the space of degree
$4$ Hodge classes) is equal to $S^2H^2(S^{[2]},\mathbb{Q})$ (resp.
is generated by $c,\,l,\,e$), while for $n\geq 3$, by the results of
de Cataldo and Migliorini \cite{deca}, the cohomology of degree $4$
of $S^{[n]}$ is generated  by $S^2H^2(S,\mathbb{Q})$ (coming from
the cohomology of $S^{(n)}$), two copies of $H^2(S,\mathbb{Q})$,
(coming via the exceptional divisor from  the cohomology of the
 codimension $2$ stratum, which has for normalization $S^{(n-2)}\times S$)
and by the classes  of the codimension $2$ subvarieties
$E_\mu$ over strata $S^{(\mu)}$ of $S$, with $l(\mu)=n-2$. If $n=3$ there is only one such
stratum and $E_\mu$ in this case is  the set of schemes of length $3$ supported over
the small diagonal. For $n\geq 4$, there are $2$ codimension $4$ strata in $S^{(n)}$
corresponding to the partitions $\{1,\ldots,1,2,2\}$ and $\{1,\ldots,1,3\}$, and
we thus get
two codimension $2$ subvarieties  $E_{\mu_1}$, $E_{\mu_2}$ in $S^{[n]}$, $n\geq 4$.

\begin{prop} \label{propcoisigen} Let $S$ be a very general projective $K3$ surface with Picard
number $1$. Then for any integer $ n$, the space of coisotropic classes of degree
$4$  on $S^{[n]}$ is generated by classes of
codimension $2$ subvarieties contained in $S_2S^{[n]}$.

\end{prop}
\begin{proof}  First of all, it is immediate to see looking at the
 proof of Theorem \ref{theoisoclass}, (iii) that for any very general $X$ hyper-K\"ahler manifold
  with $\rho(X)=2$,
there is exactly a $3$-dimensional space of isotropic classes which can be written
 as polynomials of weighted degree $2$ in $c,\,l,\,e$  (that is, the inequality
given in Theorem \ref{theoisoclass}, (iii) is an equality).

Now we first do the case $n=2$. In this case,  the algebraic
cohomology of $S^{[2]}$ for $S$ very general is given by polynomials
of weighted degree $2$ in $c,\,l,\,e$  so it suffices to exhibit
three surfaces contained in $S_2S^{[2]}$ (that is, constant cycles
surfaces)  with independent classes. Construction 1bis) gives us
such a surface, starting from a constant cycle curve $C\subset
S=\Delta_S\subset S^{(2)}$, and taking $\Sigma_1=s^{-1}(S)\subset
E\subset S^{[2]}$. The surface $\Sigma_1$ is of class $l\cdot e$.
Construction 2 gives us the surface $\Sigma_2=C^{(2)}\subset
S^{[2]}$ and clearly the class of $\Sigma_2$ is not proportional to
the class of $\Sigma_1$ because the latter is annihilated by
$s_*:H_4(S^{[2]},\mathbb{Q})\rightarrow H_4(S^{(2)},\mathbb{Q})$
while the former is not. Finally construction 4 gives us a constant
cycle surface $\Sigma_3\subset S^{[2]}$. The class of $\Sigma_3$ is
not contained in the space generated by the classes of $\Sigma_1$
and $\Sigma_2$ because the latter are annihilated by the morphism
$$p_{1*}\circ p_2^*: H^4(S^{[2]},\mathbb{Q})\rightarrow H^0(S,\mathbb{Q}),$$
where
$\Sigma\subset S\times S^{[2]}$ is the incidence subvariety and $p_1,\,p_2$ are the restrictions
to $\Sigma$ of the two projections,
while the former is not.

The general case follows
 by  analyzing  the coisotropic  classes on $S^{[n]}$ which are not polynomials in $c,\,l,\,e$.
Indeed, for $n\geq 4$ (the case $n=3$ is slightly different but can
be analyzed similarly), we observed that there are two extra degree
$4$ algebraic classes which are the classes of the varieties
$s^{-1}(S^{(\mu_{1,\ldots,1,3})})$ and
$s^{-1}(S^{(\mu_{1,\ldots,1,2,2})})$. These two codimension $2$
subvarieties are coisotropic subvarieties fibered into constant
cycle surfaces in $S^{[n]}$, hence contained in $S_2S^{[n]}$ (see construction 1)), so we
can work modulo their classes. Next, modulo these two classes, the
algebraic cohomology of $S^{[n]}$ supported on the exceptional
divisor $E$ is generated as follows: The normalization
$\widetilde{E}$ of $E$ admits a morphism $f$ to $S\times S^{(n-2)}$
and a morphism $j$ to $S^{[n]}$. Then we have the two classes
$$j_*(f^*(pr_1^*c_1(L))),\,\,j_*(f^*(pr_2^*c_1(L_{n-2})))$$
which are both classes of subvarieties of codimension $2$ of $S^{[n]}$ contained in
$S_2S^{[n]}$ because $f$ has generic fibers isomorphic to
 $\mathbb{P}^1$ and choosing a constant cycle curve $C\subset S$
which is a member of $ |L|$, $pr_1^*c_1(L)=C\times S^{(n-2)}$ is
contained in $S_1(C\times S^{(n-2)})$ and similarly for
$pr_2^*c_1(L_{n-2})$. We are thus reduced to study degree $2$
algebraic isotropic  classes on $S^{[n]}$ modulo those which are
supported on the exceptional divisor; it is immediate to see that
they are generated by polynomials in $c,\,l$ and $e$,  and we then
prove they are generated by classes of subvarieties of codimension
$2$ contained in $S^{[2]}$ starting from the case $n=2$ and applying
the sum construction.

\end{proof}
We now turn to the Chow-theoretic conjectures made in Section
\ref{sec3}. In this case, there exists  a natural splitting of the
Bloch-Beilinson filtration which is given by the de
Cataldo-Migliorini decomposition \cite{deca} and the decompositions
of the motives of the ordinary self-products or symmetric products
of $S$ given by the choice of the class $o_S\in{\rm CH}_0(S)$ as in
Section \ref{sec1}. This decomposition  is multiplicative, as proved
by Vial \cite{vial}, hence is the obvious candidate for the
Beauville decomposition in this case.

We observe first that by definition, the induced decomposition  on
${\rm CH}_0$ is the one already described in Section \ref{sec1} (see
Proposition \ref{prodecompchow}), with $o=o_S$, for which Conjecture
\ref{conjfiltopp} has been proved to hold.

Concerning the other conjectures, we deduce from the definition of the de Cataldo-Migliorini
 decomposition and from the
construction 1),   above a reduction of Conjectures  \ref{conjclasscoiso} and
\ref{conjpresquempty} to the case of
ordinary self-products or symmetric products of $S$.

Finally, we also have the following evidence for Conjecture
\ref{conjclasscoiso}, which immediately follows from the definition
of the de Cataldo-Migliorini decomposition:
\begin{lemm} For any partition $\mu$ of $n$, the codimension $i$
subvarieties $$s^{-1}(S^{(\mu)})\subset S_iS^{[n]}$$ appearing in
construction 1), with $i=n-l(\mu)$,  have their class in ${\rm
CH}^i(S^{[n]})_0$.
\end{lemm}
\subsection{The Fano variety of lines in a cubic fourfold \label{secsubfanocubic}}
It is well-known since \cite{bedo} that the
variety $F$ of lines in a smooth cubic fourfold
$W\subset \mathbb{P}^5$ is a smooth hyper-K\"ahler fourfold which is a
 deformation
of $S^{[2]}$ for a $K3$ surface
$S$  of
genus $14$ (and an adequate polarization on $S^{[2]}$).
The Chow ring of such varieties $F$ has been studied in \cite{voisinPAMQ}, confirming in particular
Conjecture \ref{conjBV} for them, and even its variant involving the Chern classes of $F$.
This variety satisfies
Conjecture \ref{conjgen}. This is implied by the following result summarizing observations
made  in \cite{voisinfano},
\cite{voisinPAMQ} (or can be obtained as an application of \cite{chapa}).
\begin{prop} \label{proramavoi}Let $W$ be a general smooth cubic fourfold, then

(a) There are rational surfaces in $F$ which can be obtained by considering the surface of lines in a
hyperplane section of $X$ with $5$ nodes.

(b) There is a  uniruled divisor in $X$  obtained as follows: $F$ admits a rational self-map
$\phi_F:F\dashrightarrow F$ which is of degree $16$.
The exceptional divisor of a desingularization $\tilde{\phi}$
of $\phi$ maps to a uniruled divisor in $F$.

\end{prop}
Note that for such an $F$, when $W$ is very general, the degree $4$ cohomology  $H^4(F,\mathbb{Q})$
is equal to $S^2H^2(F,\mathbb{Q})$ and Hodge classes of degree $4$ are linear
combinations of $l^2,\,c_2(F)$, where $l$ is the Pl\"ucker polarization.
This case is quite interesting because Shen and Vial constructed
in \cite{shenvial}  a ``Beauville decomposition''
of ${\rm CH}(F)$.  We now have the following proposition showing that
our proposal to construct a decomposition on ${\rm CH}_0(X)$, for $X$ a general
 hyper-K\"ahler manifold fits well with their results:
\begin{prop} \label{procascubic} For the Shen-Vial decomposition,  one has
$$S_1{\rm CH}_0(F)={\rm CH}_0(F)_0\oplus {\rm CH}_0(F)_2,$$
$$S_2{\rm CH}_0(F)={\rm CH}_0(F)_0.$$
In particular, $F$ satisfies Conjecture \ref{conjSiop}.
 \end{prop}
 \begin{proof}
Indeed, \cite[Proposition  19.5]{shenvial}  says the following:
 Let $\Sigma_2\subset F$ be the surface of lines $L\subset W$ such that
 there exists a $\mathbb{P}^3$ everywhere tangent to $W$ along $L$. The surface $\Sigma_2$
 is clearly the indeterminacy locus of the rational map
 $\phi: F\dashrightarrow F$ introduced above.
\begin{prop}  (Shen-Vial \cite{shenvial}) One has
\begin{eqnarray}
\label{truc5jan} {\rm CH}_0(F)_0\oplus {\rm CH}_0(F)_2={\rm
Im}\,({\rm CH}_0(\Sigma_2)\rightarrow {\rm CH}_0(F)).
\end{eqnarray}
\end{prop}
On the other hand, the Shen-Vial  decomposition is preserved by the
map $\phi_*$ and thus, if $D$ is the uniruled divisor mentioned in
Proposition \ref{proramavoi}, (b), that is, the image under the
desingularized map $\tilde{\phi}:\widetilde{F}\rightarrow F$ of the
exceptional divisor over $\Sigma_2$, one has
\begin{eqnarray}
\label{tructruc5jan}{\rm Im}\,({\rm CH}_0(\Sigma_2)\rightarrow {\rm
CH}_0(F))={\rm Im}\,({\rm CH}_0(D)\rightarrow {\rm CH}_0(F))\subset S_1{\rm CH}_0(F).
\end{eqnarray}
Finally, if $x\in F$ belongs to $S_1F$, the orbit $O_x$ contains a
curve, which has to intersect $D$, as $D$ is ample. Hence $x$ is
rationally equivalent in $F$ to a point of $D$. Thus we conclude
that
$S_1{\rm CH}_0(F)\subset {\rm Im}\,({\rm CH}_0(D)\rightarrow {\rm CH}_0(F))$
so finally
\begin{eqnarray}
\label{tructruc13jan}{\rm Im}\,({\rm CH}_0(D)\rightarrow {\rm CH}_0(F))\subset S_1{\rm CH}_0(F).
\end{eqnarray}
Combining (\ref{tructruc13jan}), (\ref{truc5jan}) and (\ref{tructruc5jan}), we get
$$S_1{\rm CH}_0(F)={\rm CH}_0(F)_0\oplus {\rm CH}_0(F)_2$$
as desired.

 The
second statement follows from the fact that the group ${\rm
CH}_0(F)_0$ of the Shen-Vial decomposition is generated by the
canonical $0$-cycle of $F$, which can be constructed (using the
results of \cite{voisinPAMQ}) by taking any $0$-cycle $o_F$ of
degree different from $0$, which can be expressed  as  a weighted
degree $4$ polynomial in $c_2(\mathcal{E})$ and $l$, where
$\mathcal{E}$ is the restriction to $F\subset G(1,5)$ of the
universal rank $2$ bundle on the Grassmannian  $G(1,5)$. However,
the rational surfaces described in Proposition \ref{proramavoi}, a)
represent the class $c_2(\mathcal{E})$ in ${\rm CH}^2(F)$. Thus
$o_F$, being supported on  a rational surface, belongs to $S_2{\rm
CH}_0(F)$. This gives the inclusion ${\rm CH}_0(F)_2\subset S_2{\rm
CH}_0(F)$ and the fact that this is an equality follows from Lemma
\ref{lechapa}, (ii) which implies that the right hand side is
isomorphic to $\mathbb{Q}$.
\end{proof}
Let us finish this section with the following result:
\begin{prop} \label{propconjvar} Conjecture \ref{conjvar} is satisfied by the variety of lines
$F$  of a cubic fourfold $W$, that is, the cycle class map is injective on
the subgroup of ${\rm CH}^2(F)$ generated by constant cycles surfaces.
\end{prop}
\begin{proof} It follows from \cite[Proposition 21.10 and Section 20]{shenvial} that
the group ${\rm CH}^2(F)$ splits as
$${\rm CH}^2(F)={\rm CH}^2(F)_0\oplus {\rm CH}^2(F)_2,$$
where ${\rm CH}^2(F)_2= {\rm CH}^2(F)_{hom}$ is the image of the map
$$I_*:{\rm CH}_0(F)_{hom}\rightarrow {\rm CH}^2(F)_{hom}$$
induced by the codimension $2$ incidence correspondence $I\subset F\times F$.
More precisely, it is proved in  
\cite[proof of Proposition 21.10]{shenvial} that
\begin{eqnarray}
\label{eq12jancubic}
I_*(g^2\sigma)=-6\sigma\,\,{\rm in}\,\,{\rm CH}^2(F)
\end{eqnarray}
for $\sigma \in {\rm CH}^2(F)_{hom}$. 
Suppose now that $\sigma\in C_2(F)$ is a combination of constant cycles
surfaces which is homologous to $0$. Then 
$g^2\sigma=0$ in ${\rm CH}_0(F)$ because $F$ satisfies Conjecture \ref{conjSiop}.
It then follows from (\ref{eq12jancubic}) that $\sigma=0$.

\end{proof}
\subsection{The case of the LLSS $8$-folds}
Let again $W$ be a cubic $4$-fold. As mentioned in the previous section,
the variety  $F:=F_1(W)$ of lines in $W$ is a smooth hyper-K\"ahler fourfold.
It is a deformation of $S^{[2]}$ for some $K3$ surfaces with adequate polarization, but
for very general $W$, it has $\rho(F)=1$.
Much more recently, Lehn-Lehn-Sorger-van Straten proved in
\cite{LLS} that starting from the variety
$F_3(W)$ of cubic rational curves in $W$, one can construct a hyper-K\"ahler
$8$-fold $Z$, which has Picard number $1$ for very general $W$, and which has been  proved in \cite{LSSsuite} to be a deformation of a hyper-K\"ahler
manifold birationally equivalent to $S^{[4]}$.
The variety $Z$ is constructed by observing first  that each cubic rational curve
$C\subset W$
moves in a $2$-dimensional linear system in the cubic surface
$S_C=<C>\cap W$, where $<C>$ is the $\mathbb{P}^3$ generated by $C$.
Finally there is a boundary  divisor   which can be contracted in the base of this $\mathbb{P}^2$-fibration
 on $F_3(W)$, and this produces the variety $Z$. Thus there is a morphism
 $q:F_3(W)\rightarrow Z$ which is birationally a $\mathbb{P}^2$-bundle.

Let us prove the following result:
\begin{prop}\label{proZ} There is a degree $6$ dominant rational map
$$\psi: F\times F\dashrightarrow Z$$
such that
\begin{eqnarray}\label{eqpourformes}\psi^*\sigma_Z=pr_1^*\sigma_F-pr_2^*\sigma_F.
\end{eqnarray}
\end{prop}
Here $\sigma_Z$, resp. $\sigma_F$ denotes the holomorphic $2$-form
of $Z$, resp. $F$.
\begin{proof} Let $L,\,L'$ be two lines in $W$ and denote by
$l$, resp $l'$ the corresponding points in $F$. Assume $l,\,l'$ are general points
of $F$; then $L$ and $L'$ are in general position in $W$ and they
generate a
$\mathbb{P}^3_{L,L'}:=<L,L'>$. The surface $S_{L,L'}:=\mathbb{P}^3_{L,L'}\cap W$
is a smooth cubic surface containing both $L$ and $L'$ and
we claim that
the linear system
$|\mathcal{O}_S(L-L')(1)|$ is a $2$-dimensional linear system
of rational cubics on $S$. This can be verified by computing
its self-intersection and intersection with $K_{S_{L,L'}}$ but it is
even easier by observing that for any choice of
point $x\in L$, the plane $<x,L'>$ intersects $L$ into one point,
and intersects $W$ along the union of  the line $L'$ and a residual  conic $C'$. Thus
we get a member of this linear system which is the union of $L$ and
of $C'$ meeting in one point: this is a rational cubic curve. Note that for each
pair $(L,L')$ we get a $\mathbb{P}^1\cong L$ of such curves.

To compute the degree of the rational map $\psi $ so constructed, we
start from a cubic surface $S\subset W$ with a $2$-dimensional
linear system of rational cubic curves. This system provides a
birational map $\tau:S\rightarrow \mathbb{P}^2$ contracting $6$
exceptional curves, which are lines in $S$. The curves in this
linear system are the pull-backs of lines in $\mathbb{P}^2$, and
they become reducible when the line passes through one of the $6$
points blown-up by $\tau$. We thus get $6$ $\mathbb{P}^1$'s of such
curves which correspond to 6 possible choices of pairs $(L,L')$,
each one giving rise to a $\mathbb{P}^1\cong L$ of reducible curves
$L\cup_x C$.

Let us finally prove formula (\ref{eqpourformes}).
In fact, we observe that according to the constructions
 of \cite{bedo} and \cite{LLS}, the $2$-forms $\sigma_F$ and $p^*\sigma_Z$ are deduced from
the choice of a generator of the $1$-dimensional vector space
$H^{3,1}(W)$ by applying the correspondences $P\subset F\times W$,
$\mathcal{C}_3\subset F_3(X)\times W$ given by the universal families of curves.
Here $p:F_3(W)\rightarrow Z$ is the forgetting morphism whose description has been sketched above.

Next we observe that the rational map
$\psi$ has a lift
$${\psi}_1:P_1\dashrightarrow F_3(W)$$
where $p_1:P\times F=P_1\rightarrow F\times F$ is the pull-back by
the first projection $F\times F\rightarrow F$ of the universal
$\mathbb{P}^1$-bundle $P\rightarrow F$, and $\psi_1$ associates to a
general triple $(l,x,l')$ with $x\in L$, the cubic curve
$$L\cup_x C,$$
where $C\subset W$ is the residual conic of $L'$ contained in the plane $<x,L'>$.
The equality of forms stated in (\ref{eqpourformes}) is then a consequence of Mumford's theorem
\cite{mumford} and the fact that
  if we restrict the universal family $\mathcal{C}_3$ to the divisor
  $D$ in
  $F_3(W)$ parameterizing reducible rational curves $C_3=L\cup C$, where $C$ is a conic
  in $X$ with  residual
   line $L'$,
then for any
$(l,x,l')\in P\times F$, the curve
$C_3$ parametrized by $\psi_1(l,x,l')$
is rationally equivalent in $W$ to $L-L'$ up to a constant.
It  follows  that
we have the equality of forms pulled-back from
$W$ via the universal correspondences:
$$p_1^*(pr_1^*\sigma_F-pr_2^*\sigma_F)=\psi_1^*(p^*\sigma_Z)
\,\,{\rm in}\,\, H^0(P_1,\Omega_{P_1}^2).$$
This immediately implies (\ref{eqpourformes}).
\end{proof}
\begin{coro} \label{corollss} The LLSS varieties $Z$ satisfy conjecture \ref{conjmain}.
\end{coro}
\begin{proof} Indeed, the variety $F$ satisfies conjecture \ref{conjmain}. This follows either from
\cite{chapa} which provides uniruled divisors and constant cycles Lagrangian
surfaces, or explicitly from Proposition \ref{proramavoi}.

Having algebraically coisotropic subvarieties $Z_1,\,Z_2$ in $F$ which are
of respective codimensions $i_1,\,i_2$ and fibered into $i_1$, resp. $i_2$-dimensional
constant cycles subvarieties of $F$, their product $Z_1\times Z_2$ is mapped by
$\psi$ onto a codimension $i_1+i_2$ subvariety of $Z$, which is fibered into
constant cycles subvarieties of dimension $i_1+i_2$.
One just has to check that $Z_1\times Z_2$ is not contracted by $\psi$, but
because both $F\times F$ and $Z$ have trivial canonical bundle,
the ramification locus of a desingularization $\tilde{\psi}:\widetilde{F\times F}\rightarrow
Z$ of $\psi$ is equal to the exceptional divisor of the birational map
$\widetilde{F\times F}\rightarrow F\times F$. So $\psi$ is of maximal rank where it is defined,
and one just has to check that $\psi$ is well defined at the general points
of the varieties $Z_1\times Z_2$ defined above.
\end{proof}
\begin{rema} {\rm There is a natural uniruled divisor in $Z$ that deserves a special study, namely
the branch locus of the desingularization
$\tilde{\psi}:\widetilde{F\times F}\rightarrow Z$ of $\psi$. This
branch locus $D$ is a non-empty divisor because $Z$ is simply
connected. It is the image of the ramification divisor of
$\tilde{\psi}$ which has to be equal to the exceptional divisor of
$\widetilde{F\times F}$ since both $F\times F$ and $Z$ have trivial
canonical divisor, and this is why $D$ is uniruled.}
\end{rema}

4 Place Jussieu, 75005 Paris, France

\smallskip
 claire.voisin@imj-prg.fr
    \end{document}